\documentclass[a4paper,12pt]{amsart}

\usepackage{fullpage}
\usepackage{enumerate, mathrsfs}
\usepackage{xcolor}

\usepackage{amsmath,amsfonts,amsthm,amssymb}
\usepackage[all]{xy}
\usepackage[T1]{fontenc}

\usepackage[colorlinks=true]{hyperref}
\usepackage{cleveref}
\usepackage{tikz}
\usetikzlibrary{shapes.geometric}
\usetikzlibrary{calc}
\usetikzlibrary{arrows}

\usepackage[colorinlistoftodos]{todonotes}
\newcommand{\ngon}[8]{ 

  \foreach \t in {1,...,#1} {
    \coordinate (#7\t) at ($#2+(#8-\t*360/#1:#3)$);
  }

  \foreach \x/\y/\z in {#4}{
    \draw[thick,\z,shorten <=#5pt, shorten >=#5pt] {(#7\x)--(#7\y)};
  }

  \setcounter{intege}{1}
  \pgfmathsetcounter{intege}{1}
  \foreach \object in {#6}{
    \node[inner sep=0pt] at (#7\theintege) {\object};
    \pgfmathsetcounter{intege}{\theintege+1}
    \setcounter{intege}{\theintege}
  }
}
\newcounter{intege}

\theoremstyle{plain}
\newtheorem{theorem}{Theorem}[section]
\newtheorem{proposition}[theorem]{Proposition}
\newtheorem{corollary}[theorem]{Corollary}
\newtheorem{lemma}[theorem]{Lemma}

\theoremstyle{definition}
\newtheorem{definition}[theorem]{Definition}
\newtheorem{example}[theorem]{Example}
\newtheorem{counterexample}[theorem]{Counterexample}
\newtheorem{conjecture}[theorem]{Conjecture}
\newtheorem{remark}[theorem]{Remark}

\crefname{counterexample}{counterexample}{counterexamples}
\Crefname{counterexample}{Counterexample}{Counterexamples}
\crefname{conjecture}{conjecture}{conjectures}
\Crefname{conjecture}{Conjecture}{Conjectures}

\definecolor{lightgrey}{rgb}{0.7,0.7,0.7}

\definecolor{darkred}{rgb}{0.7,0,0} 
\newcommand{\darkred}{\color{darkred}} 

\newcommand{\U}{\mathcal U}
\newcommand{\C}{\mathcal C}

\newcommand{\K}{\mathcal K}
\newcommand{\D}{\mathcal D}

\newcommand{\family}{\mathcal F}
\newcommand{\smallfamily}{\mathcal S}
\newcommand{\facetA}{A}
\newcommand{\facetB}{B}
\newcommand{\facetC}{C}
\newcommand{\face}{\sigma}
\newcommand{\smallface}{\tau}

\DeclareMathOperator{\link}{lk}
\DeclareMathOperator{\st}{st}
\DeclareMathOperator{\vertices}{V}

\newcommand{\defn}[1]{\emph{\darkred #1}} 
\newcommand{\set}[2]{\left\{ #1 \; : \; #2 \right\}}
\newcommand{\lkcap}[3]{\operatorname{lk}_{#1} (#2,#3)}

\subjclass[2010]{Primary 05E45; Secondary 05D05, 05C35}

\title[EKR for flag pure complexes without boundary]
{The EKR property for flag pure \\ simplicial complexes without boundary}

\date{\today}

\author[J.~Olarte]{Jorge Olarte}
\author[F.~Santos]{Francisco Santos}
\author[J.~Spreer]{Jonathan Spreer}
\author[C.~Stump]{Christian Stump}

\address[J.~Olarte and J.~Spreer]
{
Institut f\"ur Mathematik, Freie Universit\"at Berlin, Germany
}
\email{olarte@zedat.fu-berlin.de}
\email{jonathan.spreer@fu-berlin.de}

\address[F.~Santos]
{
Department of Mathematics, Statistics and Computer Science, University of Cantabria, Spain
}
\email{francisco.santos@unican.es}

\address[C.~Stump]
{
Institut f\"ur Mathematik, Technische Universit\"at Berlin, Germany
}
\email{stump@math.tu-berlin.de}

\thanks{J. Olarte, F. Santos and J. Spreer were supported by the Einstein Foundation, Berlin. C.~Stump was supported by the DFG grants STU 563/2 ``Coxeter-Catalan combinatorics'' and STU 563/4-1 ``Noncrossing phenomena in Algebra and Geometry''. F. Santos is also supported by grant MTM2014-54207-P of the Spanish  Ministry of Science,
and, while he was in residence at the Mathematical Sciences Research Institute in Berkeley, California during the Fall 2017 semester, by the Clay Institute and the National Science Foundation (Grant No. DMS-1440140)}

\keywords{%
Erd\H{o}s-Ko-Rado property,
EKR,
flag (pseudo-)manifolds, 
cluster complexes,
simplicial complexes}

\begin{document}

\begin{abstract}
  We prove that the family of facets of a pure simplicial complex $\C$ of dimension up to three satisfies the Erd\H{o}s-Ko-Rado property whenever $\C$ is  flag and has no boundary ridges. We conjecture the same to be true in arbitrary dimension and give evidence for this conjecture.
  Our motivation is that complexes with these two properties include flag pseudo-manifolds and cluster complexes.
\end{abstract}

\maketitle

\section{Introduction and main results}

In this paper we investigate the Erd\H{o}s-Ko-Rado property (EKR property in what follows) for the set of facets of a pure simplicial complex. Our starting point  is the following conjecture due to Kalai.

\begin{conjecture}[{Kalai~\cite{Kalai15MathOverflow}}]
\label{conj:triangulations}
Let~$\family$ be a family of triangulations of the $n$-gon such that every two members of~$\family$ share a common diagonal. Then~$\family$ has at most as many elements as there are triangulations of the $(n-1)$-gon.
\end{conjecture}

There is an easy way of constructing such a family~$\family$ of size \emph{exactly} the number of triangulations of the $(n-1)$-gon: choose an ear (a diagonal between two vertices at distance two) and let~$\family$ be the set of all triangulations containing it. The conjecture is that no intersecting family is larger than that.

The general EKR problem asks the following: given a family $\K$ of subsets of a ground set $\vertices$, is it true that the maximum size of an \defn{intersecting subfamily} (a subfamily $\family\subset \K$ in which no two members are disjoint) is attained by the subfamily $\family_i\subset \K$ of sets containing an element $i$, for some $i \in \vertices$? If that is the case $\K$ is said to have the \defn{EKR property}. The EKR property is \defn{strict} if every  intersecting subfamily of maximum size is of the form $\family_i$.

The original Erd\H{o}s-Ko-Rado Theorem is that the family $\K(n,r)$ of all subsets of size $r$ of a set with $n$ elements has the EKR property whenever $r\le n/2$, and the strict EKR property if $r < n/2$ (see Erd\H{o}s-Ko-Rado~\cite{EKR1961}, Hilton-Milner~\cite{HM1967}). 

\Cref{conj:triangulations} asks for the EKR property for the family of (sets of diagonals that form) triangulations of an $n$-gon.
Recall that triangulations of an $n$-gon are the facets of the \defn{simplicial associahedron}. By simplicial associahedron we mean the (flag) simplicial complex of non-crossing diagonals of an $n$-gon, that is, the independence complex of the graph whose vertices correspond to the $n(n-3)/2$ diagonals of the $n$-gon, with two diagonals forming an edge when their interiors are disjoint. It is a topological sphere of dimension $n-4$ and combinatorially dual to the usual associahedron, which is a simple $(n-3)$-polytope.
This suggests to generalize \Cref{conj:triangulations} by looking at more general flag pure simplicial complexes. 

\medskip

A simplicial complex on~$n$ vertices is a family $\C$ of subsets of~$[n]$ that is closed under taking subsets. Elements of $\C$ are called \defn{simplices}. Maximal simplices are called \defn{facets} and simplices of cardinality two are  \defn{edges}.
A complex is \defn{pure} of dimension $r-1$ (or, an \defn{$(r-1)$-complex}, for short) if all facets have size $r$ and \defn{flag} if every set of vertices that pairwise span an edge form a face. 

When the set of facets of a pure simplicial complex~$\C$ has the EKR property we call the complex \defn{pure-EKR}. Equivalently, $\C$ is pure-EKR if the maximum size of intersecting families of facets is attained at the star of some vertex; recall that the \defn{star} of a vertex $v$ is the set of faces containing $v$.
In this language the original Erd\H{o}s-Ko-Rado Theorem states that the complete complex of dimension $r-1$ with $n$ vertices is pure-EKR for $2r=n$ and strict pure-EKR for $2r<n$, and \Cref{conj:triangulations} states that the simplicial associahedron is pure-EKR.

It is easy to construct pure simplicial complexes that are not pure-EKR (see Example~\ref{exm:elementary}) but we believe that flagness plus the following relatively weak property are enough to exclude such examples:
We say that a pure simplicial complex is \defn{without boundary} if every non-maximal simplex is contained in at least two facets. 

\begin{conjecture}
\label{conj:no_free_ridges}
  Every pure, flag simplicial complex without boundary is pure-EKR.
\end{conjecture}

This conjecture can be reformulated in graph-theoretical language. For this note that 
flag simplicial complexes are exactly the \defn{independence complexes} of graphs or, equivalently, \defn{clique complexes}.
A graph is called \defn{well-covered} if its independence complex is pure.
Thus,
\Cref{conj:no_free_ridges} is equivalent to the following statement:

\smallskip

\emph{``If G is a well-covered graph such that every non-maximal independent set is contained in at least two maximal independent sets, then the family of maximal independent sets of G has the EKR property.''}

\smallskip

In this article we use the language of simplicial complexes which seems to be more convenient for our purposes; see \Cref{ssec:relations} for several more graph-theoretical articles in the literature.

\medskip

\Cref{conj:no_free_ridges} includes the following two interesting special cases, both generalizing \Cref{conj:triangulations}.

The first is  flag (pseudo)-manifolds. By a \defn{simplicial manifold} we mean a simplicial complex whose underlying set is homeomorphic to a closed manifold. By a \defn{pseudo-manifold} we mean a pure simplicial complex $\C$, such that every \defn{ridge} of~$\C$ (i.e., a codimension~$1$ face of~$\C$) is contained in \emph{exactly} two facets of $\C$ (this is sometimes also called a \defn{weak pseudo-manifold}, and an additional connectivity condition is imposed on pseudo-manifolds). 

\begin{conjecture}
\label{conj:manifolds}
  Flag simplicial manifolds are pure-EKR.
\end{conjecture}

Naturally, every simplicial manifold is a (weak) pseudo-manifold and thus the following statement generalizes \Cref{conj:manifolds}.

\begin{conjecture}
\label{conj:pseudo-manifolds}
  Flag pseudo-manifolds are pure-EKR.
\end{conjecture}

A second line of conjectures between  \Cref{conj:triangulations} and \Cref{conj:no_free_ridges} is in the context of Coxeter combinatorics.
Suppose that instead of triangulations we look at subdivisions of a polygon into subpolygons of the same size $m+2\ge 3$. This can only be done if the number of vertices is of the form $mn+2$.
The sets of diagonals producing these $(m+2)$-angulations form a pure simplicial complex of dimension $n-2$, since any $(m+2)$-angulation has $n-1$ diagonals.
More generally, for any finite root system $\Phi$ of rank $n-1$ and a positive integer $m$, Fomin and Reading~\cite{FR2005}, generalizing previous work of Fomin and Zelevinsky~\cite{FZ2003}, introduced the $m$-th \emph{generalized cluster complex} of $\Phi$ as a certain flag pure complex $\C(m,\Phi)$ of dimension $n-2$. They then showed that for root systems of type~$A_{n-1}$, the complex $\C(m,A_{n-1})$ equals the complex of $(m+2)$-angulations of an $(mn+2)$-gon.
Athanasiadis and Tzanaki~\cite{AT2008} proved that the complexes $\C(m,\Phi)$ are vertex-decomposable, hence homotopically equivalent to wedges of $(n-2)$-spheres, and gave a formula for the number of spheres in $\C(m,\Phi)$.

\medskip

We have computationally verified that for every $n,m$ with $n+m \leq 8$ and for every irreducible root system of rank~$n-1$, the complex $\C(m,\Phi)$ is pure-EKR, which leads to the following conjectures.

\begin{conjecture}
\label{conj:m_angulations}
  The complex $\C(m,A_{n-1})$ whose facets are the $(m+2)$-angulations of a convex $(mn+2)$-gon is pure-EKR.
\end{conjecture}

\begin{conjecture}
\label{conj:cluster}
  Every generalized cluster complex $\C(m, \Phi)$ is pure-EKR.
\end{conjecture}

Summing up, we have the following diagram of implications:

\begin{center}
  \begin{tikzpicture}[scale=1]
    \node (A) at (0,0)  {Conj.~\ref{conj:triangulations}};
    \node (B) at (3,1)  {Conj.~\ref{conj:manifolds}};
    \node (B2) at (6,1)  {Conj.~\ref{conj:pseudo-manifolds}};
    \node (C) at (3,-1) {Conj.~\ref{conj:m_angulations}};
    \node (D) at (6,-1) {Conj.~\ref{conj:cluster}};
    \node (E) at (9,0)  {Conj.~\ref{conj:no_free_ridges}};

    \draw[double, double equal sign distance, implies-] (A.north east) -- (B.west);
    \draw[double, double equal sign distance, implies-] (A.south east) -- (C.west);
    \draw[double, double equal sign distance, implies-] (B.east) -- (B2.west);
    \draw[double, double equal sign distance, implies-] (C.east) -- (D.west);
    \draw[double, double equal sign distance, implies-] (D.east) -- (E.south west);
    \draw[double, double equal sign distance, implies-] (B2.east) -- (E.north west);
  \end{tikzpicture}
\end{center}

Our two main results are evidence for \Cref{conj:no_free_ridges}. On the one hand, we prove the conjecture in dimensions up to three. 

\begin{theorem}
\label{thm:main}
  Pure, flag simplicial complexes without boundary of dimension at most three are pure-EKR.
\end{theorem}

This has the following immediate consequences.

\begin{corollary}
  Flag pseudo-manifolds, in particular simplicial manifolds, of dimension at most three are pure-EKR.
\end{corollary}

\begin{corollary}
  The generalized cluster complexes $\C(m, \Phi)$, for arbitrary $m$, are pure-EKR for every $\Phi$ of rank at most four.
    In particular, the complex of $(m+2)$-angulations of a convex $(mn+2)$-gon is pure-EKR for $m \geq 1$ and $n \leq 5$.
\end{corollary}

On the other hand we prove that no intersecting family in which every member touches a certain fixed edge can have size larger than the largest star:

\begin{theorem}
\label{thm:main2}
Let $\C$ be a flag pure simplicial complex without boundary and let $ab$ be an edge in $\C$.
Let~$\family$ be an intersecting family of facets of $\C$ contained in the union of the stars of~$a$ and $b$.
Then (at least) one of the two stars has size greater or equal than $|\family|$.
\end{theorem}

The $2$-dimensional case of \Cref{thm:main} (see \Cref{thm:dim2}) is a corollary of the machinery developed in \Cref{sec:setting} for proving \Cref{thm:main2}.
The $3$-dimensional case (see \Cref{theorem:3dim}) is more involved and uses additional machinery presented in \Cref{sec:three}.
The paper concludes in \Cref{sec:examples} with 
various constructions of examples of simplicial complexes which do not have the pure-EKR property.

\medskip

In the remainder of the introduction, we discuss previous work on EKR-properties and their relations to our main \Cref{conj:no_free_ridges} (see \Cref{ssec:relations}).
This is followed by some examples that limit the possibility of extending our conjectures (see \Cref{ssec:limits}).

\subsection{Relations to previous work}
\label{ssec:relations}

The following result of  Holroyd, Spencer and Talbot is further evidence in favor of \Cref{conj:no_free_ridges}:

\begin{theorem}[Holroyd, Spencer and Talbot~\cite{HST2005}]
\label{thm:HST05}
  Let~$G$ be a complete $(d+1)$-partite graph with every part of size at least~$2$.
  Then the clique complex of~$G$, which is pure of dimension $d$, is pure-EKR.
\end{theorem}

Indeed, the clique complex of every complete multipartite graph is pure and flag and the assumption that each part has size at least two forces the complex to be without boundary.
Holroyd et al.~prove also that the \defn{$r$-skeleton} (i.e., the subcomplex of faces of dimension at most~$r$) of the same class of complexes is pure-EKR for any $r \leq d$.
Note that these complexes include the cross-polytopes as minimal examples, obtained when all parts have size two. Cross-polytopes also play a role in our work; see \Cref{ex:counterex1}, \Cref{conj:crosspolytopes} and \Cref{exm:crosspolytopes}.

\medskip

The pure-EKR property, that is, the focus of this work, can be generalized in two directions (both already present in the original paper by Erd\H{o}s-Ko-Rado): the faces in the family may be required to pairwise intersect in at least $t$ points instead of a single one, or the intersecting family may consist of faces of a fixed dimension~$r$ distinct from $\dim(\C)$ (or even of various dimensions). Unfortunately the literature is not consistent, and, for instance, Borg~\cite{Bor2009} says a complex is $t$-EKR if it has the EKR property generalized in the first way, while Woodroofe~\cite{Woodroofe2011} says it is~$r$-EKR if it satisfies the property generalized in the second way. In the following discussion as well as in \Cref{sec:examples}, we use the following unifying approach.

\begin{definition}
\label{EKR-general}
A \defn{$t$-intersecting family} is a family in which every two elements have at least $t$ common members.
A complex~$\C$ is called \defn{$t$-intersecting EKR} if every $t$-intersecting family of faces has at most as many members as the star of some $(t-1)$-face, with the property being \defn{strict} if all families maximizing the number of members are stars. It is called \defn{$t$-intersecting pure-EKR} if the same holds for families of facets (in this case~$\C$ is assumed to be pure). When $t=1$ we omit the word $t$-intersecting.
\end{definition}

In this language, the classical paper Erd\H{o}s-Ko-Rado~\cite{EKR1961} shows that $\K(n,r)$ is pure-EKR for $n \geq 2r$ and $t$-intersecting pure-EKR for~$n$ sufficiently large.
The following more precise version is due to Ahlswede-Khachatrian~\cite{AK1997}, following work by Frankl~\cite{Fra1978} and Wilson~\cite{Wil1984}.
We refer to Borg~\cite[Theorem~1.3]{Bor2009} for details.

\begin{theorem}
\label{thm:EKR2}
  Let $1 \leq t < r < n$ and let~$\K(n,r) $ be the complete $(r-1)$-complex on $n$ vertices. Then
  \begin{itemize}
    \item~$\K(n,r) $ is $t$-intersecting pure-EKR if and only if $(t+1)(r-t+1) \leq n$ and
    \item~$\K(n,r) $ is strict $t$-intersecting pure-EKR if and only if $(t+1)(r-t+1) < n$.
  \end{itemize}
\end{theorem}

Holroyd and Talbot conjectured the (strict) EKR property of low-dimensional subcomplexes of flag complexes.

\begin{conjecture}[Holroyd and Talbot~\cite{HT2005}]
\label{conj:HT}
  Let~$\C$ be a flag simplicial complex with minimal facet cardinality~$n$ and let $\C^{(r)}$ be its $(r-1)$-skeleton.
  Then
  \begin{enumerate}[(1)]
    \item $\C^{(r)}$ is pure-EKR if $2r \leq n$ and \label{eq:conj:HT1}
    \item $\C^{(r)}$ is strict pure-EKR if $2r < n$.
  \end{enumerate}
\end{conjecture}

Observe that here the conditions that all facets of~$\C$ have cardinality (much) larger than~$r$ implies the $(r-1)$-skeleton to be pure and without boundary.
For the case where~$\C$ itself is not only flag, but also pure and without boundary, we see in \Cref{prop:meep,prop:meep2} that $\C^{(r)}$ has what we call the \defn{missing edge exchange property} (see \Cref{def:meep}).
In particular, our \Cref{theorem_2_to_1_meep}, which is a generalization of \Cref{thm:main2} for complexes with that property, may be a step towards proving \Cref{conj:HT}.

In~\cite{Bor2009}, Borg removed flagness from \Cref{conj:HT} and proved several special cases of the statement. For example, the following is a special case of~\cite[Theorem~2.1]{Bor2009}.

\begin{theorem}[{Borg~\cite{Bor2009}}]
\label{thm:Borg}
  Let $t \leq r$ and set $n^*_0(r,t) = (r-t)\binom{3r-2t-1}{t+1}+r$.
  Let~$\C$ be a simplicial complex with minimal facet cardinality at least~$n^*_0(r,t)$ and let $\C^{(r)}$ be its $(r-1)$-skeleton.
  Then $\C^{(r)}$ is strict $t$-intersecting pure-EKR.
\end{theorem}

Finally, 45 years ago Chv{\'a}tal conjectured the following very general EKR property for arbitrary simplicial complexes.
Chv{\'a}tal's Conjecture implies \Cref{conj:HT}\eqref{eq:conj:HT1} and Borg's relaxation, while it does not immediately imply any pure-EKR properties such as \Cref{thm:Borg} or our \Cref{conj:no_free_ridges}.

\begin{conjecture}[Chv{\'a}tal~\cite{Chv1974}]
  \label{chvatal}
  Every simplicial complex is EKR.
\end{conjecture}

\subsection{Limits of generalizations}
\label{ssec:limits}

The following are two elementary prototypes of counterexamples that emphasize the importance of the two conditions of \emph{flagness} and \emph{absence of boundary} in \Cref{conj:no_free_ridges}.

\begin{counterexample}[Two elementary non-pure-EKR complexes]
\label{exm:elementary}
  The boundary complex of the triangle or, more generally, of a $d$-simplex for $d\ge 2$, is the prototype of a non-flag complex.
  It is not pure-EKR as all $d+1$ facets meet pairwise, while each vertex only belongs to only $d$ of them.
  Among flag complexes, a manifold with boundary which is not pure-EKR is depicted below. It can easily be generalized to higher dimensions by considering a $d$-simplex and its $d+1$ neighbors:

  \begin{center}
    \begin{tikzpicture}
      \node[draw,
        regular polygon,
        regular polygon sides=3,
        fill=black!20,
        minimum height=8em] (a) {};
      \foreach \x in {1,2,...,3}
        \fill (a.corner \x) circle[radius=2pt];
      \node[draw,
        shape border rotate=180,
        regular polygon,
        regular polygon sides=3,
        fill=black!40,
        minimum height=4em] (a) {};
      \foreach \x in {1,2,...,3}
        \fill (a.corner \x) circle[radius=2pt];
    \end{tikzpicture}
  \end{center}
\end{counterexample}

In \Cref{sec:examples} we discuss several ways to extend these prototypes to more general families of counterexamples.

\medskip

We are also interested in whether or not our conjectures can be generalized to \emph{strict pure-EKR properties}, or \emph{higher $t$-intersecting pure-EKR properties}.

\medskip

Concerning \Cref{conj:manifolds}, the answer is that we can neither hope for a strict pure-EKR property nor a~$2$-intersecting pure-EKR property as can both be seen in the boundary complexes of cross polytopes:

\begin{counterexample}[Strict pure-EKRness for flag manifolds]
\label{ex:counterex1}
  The boundary complex of the octahedron is a flag manifold without boundary that is  not strict pure-EKR: all vertices are contained in~$4$ facets, but any of the $2^4$ ways of choosing one of each pair of opposite facets is also an intersecting family. The same happens in the higher-dimensional analogue, the boundary complex $\partial \beta^d$ of the $d$-dimensional cross-polytope. There are $2^{2^{d-1}}$ intersecting families of facets of the same size as the $2^d$ stars of vertices: every family not containing a pair of opposite facets is intersecting.
\end{counterexample}

\begin{counterexample}[$2$-intersecting pure-EKRness for flag manifolds]
\label{ex:counterex2}
  The boun\-dary complex $\partial \beta^4$ of the $4$-dimensional cross polytope $\beta^4$ is a flag manifold without boundary that does not have the~$2$-intersecting pure-EKR property: all edges are contained in exactly~$4$ facets, but any facet together with its four neighbors form a~$2$-intersecting family of size five.
\end{counterexample}

We extend \Cref{ex:counterex1,ex:counterex2} to more general families of counterexamples in \Cref{sec:examples}.

\medskip

We next have a closer look at the limits of generalizations towards triangulations.
The computations discussed before \Cref{conj:m_angulations} suggest that the pure-EKR property in \Cref{conj:triangulations,,conj:m_angulations,,conj:cluster} might be strict.
Moreover, the complexes $\C(m,A_{n-1})$ are also~$2$-intersecting pure-EKR for all $n+m \leq 8$.
However, the following closely related examples show that they are neither strict $2$-inter\-sec\-ting pure-EKR nor $3$-intersecting pure-EKR in general.

\begin{counterexample}[strict $2$-intersecting pure-EKRness for triangulations]
\label{ex:counterex4}
  Consider triangulations of a heptagon, all of which have four internal diagonals.
  Any family consisting of one triangulation~$T$ and the four triangulations obtained by flipping an inner diagonal in~$T$ is a $2$-intersecting family of size five which is not a star, e.g.,

  \vspace{5pt}

  \begin{center}
    \begin{tikzpicture}[scale=.5, every node/.style={scale=0.7}]
      \ngon{7}{(-12.5,0)}{2}
          {1/2/black,2/3/black,3/4/black,4/5/black,5/6/black,6/7/black,1/7/black,
          1/3/lightgrey,1/4/lightgrey,1/5/lightgrey,1/6/lightgrey}
        {6}{1,2,3,4,5,6,7}{first}{112.5}

      \ngon{7}{(-7.5,0)}{2}
          {1/2/black,2/3/black,3/4/black,4/5/black,5/6/black,6/7/black,1/7/black,
          2/4/lightgrey,1/4/lightgrey,1/5/lightgrey,1/6/lightgrey}
        {6}{1,2,3,4,5,6,7}{first}{112.5}

      \ngon{7}{(-2.5,0)}{2}
          {1/2/black,2/3/black,3/4/black,4/5/black,5/6/black,6/7/black,1/7/black,
          1/3/lightgrey,3/5/lightgrey,1/5/lightgrey,1/6/lightgrey}
        {6}{1,2,3,4,5,6,7}{first}{112.5}

      \ngon{7}{(2.5,0)}{2}
          {1/2/black,2/3/black,3/4/black,4/5/black,5/6/black,6/7/black,1/7/black,
          1/3/lightgrey,1/4/lightgrey,4/6/lightgrey,1/6/lightgrey}
        {6}{1,2,3,4,5,6,7}{first}{112.5}

      \ngon{7}{(7.5,0)}{2}
          {1/2/black,2/3/black,3/4/black,4/5/black,5/6/black,6/7/black,1/7/black,
          1/3/lightgrey,1/4/lightgrey,1/5/lightgrey,5/7/lightgrey}
        {6}{1,2,3,4,5,6,7}{first}{112.5}
    \end{tikzpicture}
  \end{center}

  \vspace{5pt}

  In contrast, the star of a face of size two cannot have size larger than five: once we introduce two diagonals into the heptagon we are either left with two quadrilaterals (which yields four possibilities to complete to a triangulation) or with a pentagon (which yields five).
\end{counterexample}

\begin{counterexample}[$3$-intersecting pure-EKRness for triangulations]
\label{ex:counterex3}
  Consider triangulations of an octagon, all of which have five internal diagonals.
  The five triangulations in \Cref{ex:counterex4} joined with the edge $1$---$7$ form a $3$-intersecting family.
  Since the starting triangulation shares~$4$ edges with all others, one can as well flip the additional edge $1$---$7$ to obtain another triangulation, producing a $3$-intersecting family of size six in total, e.g.,

  \vspace{5pt}

  \begin{center}
    \begin{tikzpicture}[scale=.5, every node/.style={scale=0.7}]
      \ngon{8}{(-12.5,0)}{2}
          {1/2/black,2/3/black,3/4/black,4/5/black,5/6/black,6/7/black,7/8/black,1/8/black,
          1/3/lightgrey,1/4/lightgrey,1/5/lightgrey,1/6/lightgrey,1/7/lightgrey}
        {6}{1,2,3,4,5,6,7,8}{first}{112.5}

      \ngon{8}{(-7.5,0)}{2}
          {1/2/black,2/3/black,3/4/black,4/5/black,5/6/black,6/7/black,7/8/black,1/8/black,
          2/4/lightgrey,1/4/lightgrey,1/5/lightgrey,1/6/lightgrey,1/7/lightgrey}
        {6}{1,2,3,4,5,6,7,8}{first}{112.5}

      \ngon{8}{(-2.5,0)}{2}
          {1/2/black,2/3/black,3/4/black,4/5/black,5/6/black,6/7/black,7/8/black,1/8/black,
          1/3/lightgrey,3/5/lightgrey,1/5/lightgrey,1/6/lightgrey,1/7/lightgrey}
        {6}{1,2,3,4,5,6,7,8}{first}{112.5}

      \ngon{8}{(2.5,0)}{2}
          {1/2/black,2/3/black,3/4/black,4/5/black,5/6/black,6/7/black,7/8/black,1/8/black,
          1/3/lightgrey,1/4/lightgrey,4/6/lightgrey,1/6/lightgrey,1/7/lightgrey}
        {6}{1,2,3,4,5,6,7,8}{first}{112.5}

      \ngon{8}{(7.5,0)}{2}
          {1/2/black,2/3/black,3/4/black,4/5/black,5/6/black,6/7/black,7/8/black,1/8/black,
          1/3/lightgrey,1/4/lightgrey,1/5/lightgrey,5/7/lightgrey,1/7/lightgrey}
        {6}{1,2,3,4,5,6,7,8}{first}{112.5}

      \ngon{8}{(12.5,0)}{2}
          {1/2/black,2/3/black,3/4/black,4/5/black,5/6/black,6/7/black,7/8/black,1/8/black,
          1/3/lightgrey,1/4/lightgrey,1/5/lightgrey,1/6/lightgrey,6/8/lightgrey}
        {6}{1,2,3,4,5,6,7,8}{first}{112.5}
    \end{tikzpicture}
  \end{center}

  \vspace{5pt}

  In contrast, the star of a face of size three cannot have size larger than five.
\end{counterexample}

\section{A reduction in arbitrary dimension}
\label{sec:setting}

\subsection{Simplicial complexes}
\label{subsec:basic}
We start by recalling the necessary elementary definitions on (abstract) simplicial complexes, sometimes also referred to in the literature as \defn{downsets}, \defn{hereditary sets}, or \defn{independence systems}.

\medskip

A \defn{simplicial complex}~$\C$ is a family of subsets of some ground set~$S$ such that
\[
  \smallface \subseteq \face \in \C \Rightarrow \smallface \in \C.
\]
The elements of~$\C$ are called \defn{faces} or \defn{simplices}, the \defn{dimension} of a face~$\face$ is given by $\dim(\face) = |\face|-1$, that is, by its size minus one, and the \defn{dimension} of~$\C$ is the maximal dimension of one of its faces. 
We sometimes call a simplicial complex of dimension~$d$ a \defn{$d$-complex}. A subset of $S$ not contained in $\C$ is termed a \defn{non-face} of $\C$.
The containment-wise maximal faces are called \defn{facets}, and the faces of dimension zero, one, two and three are called \defn{vertices}, \defn{edges}, \defn{triangles} and \defn{tetrahedra} respectively. The set of vertices of a simplicial complex~$\C$ is denoted by $\vertices (\C)$.
We usually denote simplices by Greek letters such as $\smallface$ and $\face$, facets by capital Latin letters such as $\facetA$, and vertices by small Latin characters such as~$a$, $b$ or~$v$. Whenever we denote a face by its elements from the ground set, say $\face = \{a,b,c\} \in \C$, we omit set notation and write $\face = abc$ for brevity.
For a face~$\face$ of~$\C$, define the \defn{complement} by
\[
  \C \setminus \face := \set{ \smallface \in \C}{ \smallface \cap \face = \emptyset},
\]
the \defn{link} by
\[
  \link_\C(\face) := \set{ \smallface \in \C}{ \face \cap \smallface = \emptyset,\ \face \cup \smallface \in \C},
\]
and the \defn{star} by
\[
  \st_\C(\face) := \set{ \smallface }{\face \subseteq \smallface \in \C}.
\]

The complex~$\C$ is called \defn{flag} if all its containment-wise minimal non-faces have cardinality two.
It is \defn{pure} if all of its facets have the same dimension, say $d$, in which case faces of dimension $d-1$ are called \defn{ridges}. 

\subsection{Dual pairs of cross-intersecting families}

The following definitions are of central importance for the proofs of the two main results of this article.

\begin{definition}
\label{def:dp}
Let~$\D$ be a simplicial complex and let $\U$ be a subset of faces of~$\D$. We say that:
\begin{enumerate}

\item $\U$ is an  \defn{upper set} if $\face\in \U$ and $\face\subseteq \face'\in \D$ implies $\face'\in \U$. We denote by $\langle \U \rangle$ the upper set generated by $\U$; that is:
\[
\langle \U \rangle := \set{ \face'\in \D }{ \exists \, \face\in \U, \face\subseteq\face'}.
\]
\item The \defn{essential vertices $E(\U)$} of $\U$ are the vertices in the minimal elements of $\U$.
\item Another subset $\U'$ of faces of~$\D$ \defn{cross intersects} $\U$ if for all pairs of faces $\face\in \U$ and $\face'\in \U'$ we have $\face\cap \face'\ne\emptyset$.
\item The \defn{intersecting dual} of $\U$ in~$\D$ is
\[
\U^* :=\set{\face'\in \D }{ \face'\cap \face\ne \emptyset, \forall \face\in \U}.
\]
\item We  say that $(\U_1,\U_2)$ is a \defn{dual pair} in~$\D$ if $\U_1={\U_2}^*$ and $\U_2={\U_1}^*$.

\end{enumerate}
\end{definition}

\begin{proposition}
\label{prop:dual_pairs}
Let~$\D$ be a simplicial complex and let $\U$ be a set of faces of~$\D$. Then the following basic facts on cross intersecting pairs hold with respect to $\U$.

\begin{enumerate}
\item $\U^*$ is the largest family of faces of~$\D$ cross intersecting $\U$, and it is an upper set.
\label{prop:dual_pairs.1}
\item $\langle\U\rangle\subseteq \U^{**}$ and $\U^*= \U^{***}$. In particular, $(\U^*, \U^{**})$ is a dual pair for every $\U$, and every dual pair arises in this way.
\label{prop:dual_pairs.2}
\item $E(\U^*)\subseteq E(\U)$. In particular, if $\U=(\U_1,\U_2)$ is a dual pair then $E(\U_1)=E(\U_2)$, and we denote this set by $E(\U)$.
\label{prop:dual_pairs.3}
\item A dual pair cannot be properly contained in another dual pair. I.e., if $\U = (\U_1, \U_2)$ and $\U' = (\U'_1, \U'_2)$ are dual pairs with $U_1 \subset U'_1$ and $U_2 \subset U'_2$ then $U_1 = U'_1$ and $U_2 = U'_2$.
\label{prop:dual_pairs.4}
\end{enumerate}
\end{proposition}

\begin{proof}
We prove \Cref{prop:dual_pairs} part by part.

\begin{description}
  \item[\eqref{prop:dual_pairs.1}] This part is obvious.

  \item[\eqref{prop:dual_pairs.2}] Since $\langle \U\rangle$ cross intersects $\U^*$, we can apply part \eqref{prop:dual_pairs.1} to $\U^*$ to obtain $\langle \U\rangle \subseteq \U^{**}$. In particular, we also have $\U^* \subseteq \U^{***}$. On the other hand, intersecting every element in $\U^{**}$ is a stronger condition than intersecting every element in $\U$, hence $\U^{***} \subseteq \U^*$.

  \item[\eqref{prop:dual_pairs.3}] If $E(\U^*)=\emptyset$ then either $\U^*=\emptyset$ or $\U^*=\langle \emptyset\rangle = \D$, and the result follows (observe that $\emptyset^*=\D$ and $\D^*=\emptyset$). Hence, let $v\in E(\U^*)$. Then there exists a minimal element $\face \in \U^*$ such that $v\in \face$. Then $\face\setminus v \notin \U^*$ and there exists $\smallface \in \U$ such that $\smallface \cap \face = v$. Let $\smallface'$ be any minimal element of $\U$ contained in $\smallface$. As $\smallface'$ intersects $\face$, necessarily $v\in \smallface'$, hence $v\in E(\U)$.

  \item[\eqref{prop:dual_pairs.4}] Suppose for dual pairs $\U = (\U_1, \U_2)$ and $\U' = (\U'_1, \U'_2)$ we have that $U_1 \subsetneq U'_1$. Hence, we have for the dual that $U'_2 \subseteq U_2$. If $U'_2 = U_2$, then by duality $U_1 = U'_1$. Hence $U'_2 \subsetneq U_2$. In particular, $\U$ as a dual pair cannot be properly contained in another dual pair $\U'$. \qedhere
\end{description}
\end{proof}

Observe that if $\U_1$ and $\U_2$ cross intersect, then the dual pair $(\U_1^{**}, \U_1^* )$ satisfies $\U_1\subseteq \U_1^{**}$ and $\U_2\subseteq \U_1^*$.
Part \eqref{prop:dual_pairs.1} of the proposition suggests the following definition:

\begin{definition}
  Let $\U= (\U_1,\U_2)$ be a dual pair for a complex~$\D$ and $\U'=(\U'_1,\U'_2)$ be a dual pair for a complex $\D'$. We say that $\U$ and $\U'$ are \defn{isomorphic} if there exists a bijection $\phi:E(\U)\to E(\U')$ such that, up to interchanging $\U_1'$ and $\U_2'$, a face $\face$ is a minimal element in $\U_i$ if and only if $\phi(\face)$ is a minimal element in $\U_i'$.
\end{definition}

In particular, this definition means that $(\U_1,\U_2)$ and $(\U_2,\U_1)$ are isomorphic.
Moreover, observe that we do not assume~$\D$ and $\D'$ to be isomorphic complexes or to have the same number of vertices.

\begin{example}
  \label{ex:dual0}
  For every complex~$\D$,  $(\emptyset, \D)$ is a dual pair, which we call the trivial dual pair. If $\face=\{v_1,\dots.v_k\}$ is a face of $\U$ then $(\langle\face\rangle,\langle v_1,\dots , v_k \rangle)$ is a dual pair too. As a special case,  $(\langle v\rangle, \langle v\rangle)$ is a  dual pair for every vertex~$v$ in~$\D$.\footnote{Here and in what follows we omit braces for the set of faces inside the angle bracket $\langle \cdot \rangle$ for uppersets. Moreover, recall that we write a face as a list of its vertices. That is, the upper set generated by the $1$-element set $\{\face\}=\{\{v_1,\dots,v_k\}\}$ is denoted by $\langle \face\rangle = \langle v_1v_2\dots v_k\rangle$. Do not confuse with $\langle v_1,\dots , v_k \rangle$, the upper set generated by the set of $k$ vertices $\{ \{v_1\}, \dots, \{v_k\}\}$.}
\end{example}

\begin{example}
  \label{ex:dual1}
  If~$\D$ is $1$-dimensional and a cycle of length greater than four, the three types of dual pairs mentioned in \Cref{ex:dual0} are the only ones that can occur. If~$\D$ is a cycle of length four (with vertices $u,v,w,x$ in that order) a new type arises, such that the full list is:
  \[
    (\emptyset, \D),\qquad (\langle v\rangle, \langle v\rangle), \qquad
    (\langle u,v\rangle, \langle uv\rangle) \qquad (\langle uv,wx\rangle,\langle vw,ux\rangle).
  \]
\end{example}

\begin{example}
  \label{ex:dual2}
  Let~$\D$ be the~$d$-simplex and, for each $k$, let $\U^{(k)}$ denote the set of its $k$-dimensional faces. Then $\U^{(k)}$ is dual to $\U^{(d-k)}$. In particular, if $d$ is even then $\U^{(d/2)}$ is self-dual, that is $(\langle \U^{(d/2)} \rangle , \langle \U^{(d/2)} \rangle)$ is a dual pair.
\end{example}

\begin{example}
  \label{ex:dual3}
  A full list of dual pairs generated by vertices and edges in arbitrary complexes $\D$ is listed below.

  \begin{figure}[ht]
    \includegraphics[height=3cm]{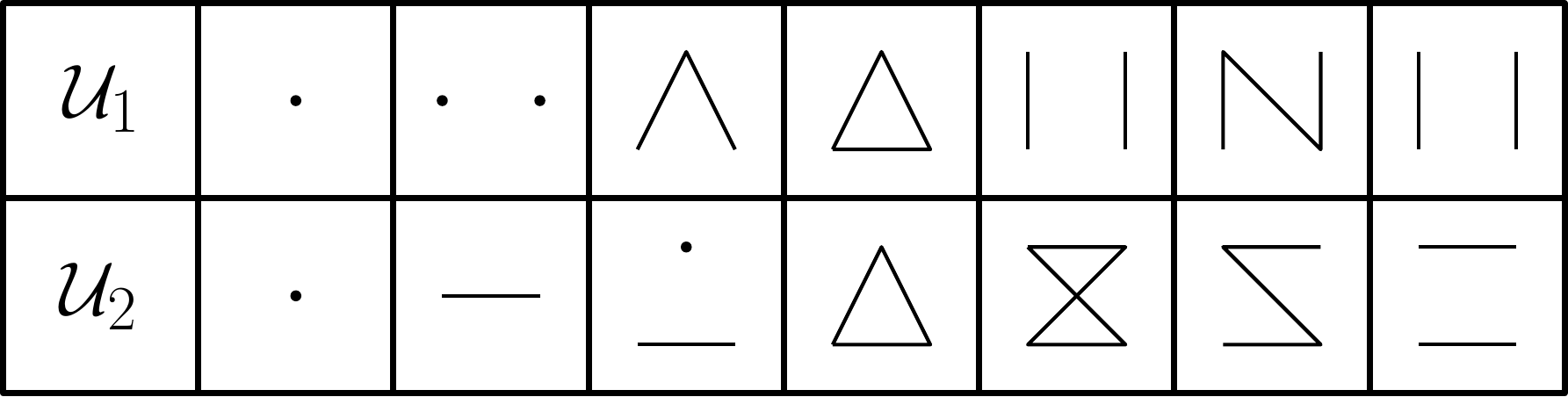}
    \caption{The minimal elements of all possible non trivial isomorphism classes of dual pairs generated by vertices and edges.\label{fig:dualpairs}}
  \end{figure}

  To see that there are no others, first assume that the ambient complex $\D$ contains a complete graph (every pair of vertices forms an edge).

  Dual pairs involving vertices on at least one side occur in exactly three scenarios (up to isomorphism): Since there cannot be more than two connected components per side, we have a single vertex, two vertices, or a vertex and a simple graph in which every vertex is contained in every edge (that is, a single edge).

  It remains to consider dual pairs having only edges on both sides. Both graphs formed by these edges must have vertex covering sets of cardinality two, and no vertex cover of size one. Moreover, there cannot be any vertex of degree three, because every edge in the dual must then contain this vertex (otherwise there is a triangle in the dual) and thus the double dual then contains a vertex. We conclude by noting that there are only four graphs admitting vertex covers of size two with maximum vertex degree at most two, giving rise to three more dual pairs.

  Now assume that the ambient space does not contain a complete graph. Note that removing edges from one of the previously listed dual pairs results in dual pairs isomorphic to dual pairs already in the list. The only exception to this rule is when we remove both of the diagonals of the sixth dual pair in \Cref{fig:dualpairs}. This results in the seventh isomorphism class, which thus is the only isomorphism class that can not be realized in a complex with a complete graph.
\end{example}

In order to apply the framework of dual pairs to intersecting families of facets of pure simplicial complexes we first need to establish some additional notation.

Let~$\family$ be an intersecting family of facets in a pure (but not necessarily flag) complex~$\C$.
For a vertex $a \in \C$, we define $\family_a := \st_\C (a) \cap \family$ to denote the set of all facets in~$\family$ containing~$a$. Similarly, we define $\family_\smallface := \st_\C (\smallface) \cap \family $. For vertices $a,b \in \C$ we write
$\family_{a\overline b} := \family_a \setminus \st_\C (b)$ to denote all facets in~$\family$ containing~$a$ but not $b$. 

Furthermore, let
$\lkcap{\C}{a}{b} := \set{\face \in \C}{\vertices (\face) \subset \vertices (\link_\C (a) \cap \link_\C (b))}$
be the complex of faces spanned by vertices in the intersection of the vertex links of~$a$ and $b$ in~$\C$.
Consider the restriction
\[
\family_{a\overline{b}} |_{\lkcap{\C}{a}{b}} := \set{ \facetA \cap \vertices(\lkcap{\C}{a}{b}) }{ \facetA \in \family_{a\overline b}}.
\]

The collection $\family_{a\overline{b}} |_{\lkcap{\C}{a}{b}}$ contains all the relevant information about the conditions that $\family_a$ imposes to elements in $\family_b$ in order for~$\family$ to be an intersecting family: If $\facetA\in \C_{b\overline a}$, then $\facetA$ intersects all elements in $\family_a$ if and only if it intersects every element in $\family_{a\overline{b}} |_{\lkcap{\C}{a}{b}}$.
In particular, $\family_{a\overline{b}} |_{\lkcap{\C}{a}{b}}$  and $\family_{b\overline{a}} |_{\lkcap{\C}{a}{b}}$ are cross intersecting pairs.

\begin{definition}
Let $\U=(\U_a,\U_b)$ be a dual pair of $\lkcap{\C}{a}{b}$. We say that $\U$
\defn{supports the intersecting family~$\family$ at~$a$ and $b$}
if $\family_{a\overline{b}} |_{\lkcap{\C}{a}{b}} \subseteq \U_a$ and $\family_{b\overline{a}} |_{\lkcap{\C}{a}{b}} \subseteq \U_b$.
\end{definition}

Note that, by definition, $(\U_a,\U_b)$ possibly supports an intersecting family larger than~$\family$, that is, a ``strictly better candidate'' for a counterexample to the pure-EKR property. Naturally, it suffices to look at these underlying larger families to establish that~$\C$ is pure-EKR.

To see this, note that \Cref{prop:dual_pairs} implies that both of the following dual pairs support~$\family$ at $ab$:
\[
\left({\family_{a\overline{b}} |_{\lkcap{\C}{a}{b}}}^{**},{\family_{a\overline{b}} |_{\lkcap{\C}{a}{b}}}^*\right) \enspace \text{ and } \enspace \left({\family_{b\overline{a}} |_{\lkcap{\C}{a}{b}}}^*,{\family_{b\overline{a}} |_{\lkcap{\C}{a}{b}}}^{**}\right).
\]
However, these may not be necessarily equal: Suppose $\family_{a\overline{b}} |_{\lkcap{\C}{a}{b}}=\{ u\}$ and  $\family_{b\overline{a}} |_{\lkcap{\C}{a}{b}}=\{uv\}$. Then
${\family_{a\overline{b}} |_{\lkcap{\C}{a}{b}}}^*={\family_{a\overline{b}} |_{\lkcap{\C}{a}{b}}}^{**}=\langle u\rangle$ and ${\family_{b\overline{a}} |_{\lkcap{\C}{a}{b}}}^*= \langle u,v\rangle$, ${\family_{b\overline{a}} |_{\lkcap{\C}{a}{b}}}^{**}= \langle uv\rangle$.
Note that for every $v\in E(\U)$ there always exists $\face \in \U_a$ and $\smallface \in \U_b$ such that $\face\cap\smallface = v$.

\subsection{Base faces of size two}
\label{ssec:proofarbdim}

The underlying idea for the proof of both our main results is to consider a face~$\face$ intersecting every element of a given intersecting family~$\family$.
Observe that every facet in~$\family$ is such a face, because~$\family$ is an intersecting family.
If some vertex~$a$ used in~$\family$ is as well such a face, then the family~$\family$ must be contained in the star of~$a$ and~$\family$ cannot be a counterexample to the pure-EKR property. 

The proof of our first main result \Cref{thm:main} goes by interpolating between these two cases, the one we want and the want we trivially have. Incidentally, our second main result \Cref{thm:main2}, proved in this section in arbitrary dimension, serves as one step in this process in both the $2$- and the $3$-dimensional setting.

\begin{definition}
  Let~$\family$ be an intersecting family of facets in a pure~$d$-complex~$\C$.
  We say that a face $\face \in \C$ is a \defn{base face} for~$\family$ if $\face$ intersects all facets in~$\family$.
  That is, if~$\family$ is contained in the union of stars of vertices of $\face$.
\end{definition}

\Cref{thm:main2} deals with the case when we have a base face of size two.
Our proof works for pure complexes more general than flag and without boundary. We now introduce the exact property that we need. We believe that this property captures some very essential characteristics of pure-EKR simplicial complexes. 

\begin{definition}
  \label{def:meep}
  A~$d$-dimensional pure simplicial complex~$\C$ is said to have the \defn{missing edge exchange property}, if for every ridge $\face \in \C$ and every vertex $v$ such that $\face \cup v$ is a facet there is a vertex $u\ne v$ such that $\face \cup u$ is again a facet and $uv$ is not an edge in~$\C$.
\end{definition}

Observe that  pure complexes with the missing edge exchange property have no boundary. The converse is not true in general, but holds for flag complexes:

\begin{proposition}
  \label{prop:meep}
  Every flag pure simplicial complex without boundary has the missing edge exchange property.
\end{proposition}

\begin{proof}
  Let~$\C$ be a flag pure simplicial complex without boundary. Because~$\C$ has no boundary, for every facet $\facetA \in \C$ and for every ridge $\face \in \facetA$, there exists another facet $\facetB \in \C$ containing $\face$. Moreover, let $u\in \facetA$ and $v \in \facetB$ be the vertices opposite $\face$. Then, $uv$ cannot be an edge in~$\C$ because otherwise~$\C$ contains a $(d+2)$-clique, which contradicts the flagness of~$\C$.
\end{proof}

A nice and interesting feature of the missing edge exchange property is that it is preserved under taking lower dimensional skeleta.

\begin{proposition}
  \label{prop:meep2}
  Let~$\C$ be a pure simplicial complex with the missing edge exchange property, and let~$\D$ be the $k$-skeleton of~$\C$, $ k < d$. Then~$\D$ has the missing edge exchange property.
\end{proposition}

\begin{proof}
  To see this let $\face \in \D$ to be a $k$-face of~$\D$, and let $v \in \face$. Then there exists a facet $\facetA \in \C$ containing $\face$. By the missing edge exchange property of~$\C$, there exists a vertex $u \in \C$ such that $\facetB = \facetA \setminus \{ v\} \cup \{ u \}$ is a facet, and $uv$ is not an edge of~$\C$ (nor~$\D$). In particular, $\smallface = \face \setminus \{v\} \cup \{u\}$ is a $k$-face of~$\D$. It follows that~$\D$ has the missing edge exchange property.
\end{proof}

The following lemma shows that the missing edge exchange property can be a key ingredient in showing that complexes are pure-EKR in a very general context.
It is our most powerful tool for the rest of the paper. 

\begin{lemma}
\label{lemma_vertex_flip}
Let~$\C$ be a~$d$-complex with the missing edge exchange property, and
let~$\family$ be an intersecting family of facets of~$\C$ with base face $\face \supseteq ab$. Let $\U = (\U_a,\U_b)$ be a dual pair that supports~$\family$ at $ab$ and let $v \in E(\U)$ such that $v\not\in\face$ and there exists no $c\in \face\setminus \{a,b\}$ with $cv\in \C$.

Then there exists an intersecting family $\family'$ and a dual pair $\U'$ supporting $\family'$ at $ab$, such that $|\family'|\geq |\family|$ and $E(\U') \subseteq E(\U)\setminus\{v\}$.
\end{lemma}

\begin{proof}
Let
\[
Q_a(v) = \set{\lambda \in \U_a }{ \exists \, \smallface \in \U_b, \enspace \smallface \cap \lambda = \{v\} }
\]
and let
\begin{eqnarray*}
R_a(v)&=&
\set{\facetA \in \family_{a\overline b}}{ \exists \, \smallface \in \U_b \enspace \smallface\cap \facetA = \{v\}} \\
&=&
\set{\facetA \in \family_{a\overline b}}{ \facetA\cap \vertices(\lkcap{\C}{a}{b}) \in Q_a}.
\end{eqnarray*}
$Q_a(v)$ is a set of faces of $\lkcap{\C}{a}{b}$, while $R_a(v)$ is a set of facets of~$\C$.
While $R_a(v)$ can be empty, $Q_a(v)$ is not because~$v$ is essential.
Define $Q_b(v)$ and $R_b(v)$ similarly. W.l.o.g. assume $|R_a(v)| \geq |R_b(v)|$. The idea of the proof is to replace the facets in $R_b(v)$ by facets
in the link of~$a$ that are not in~$\family$ (and, in particular, not in $\family_{a\overline b}$).

By the missing edge exchange property of~$\C$, we have for each $\facetA \in R_a(v)$ that there exists a vertex $u$ such that $\facetA' = \facetA\setminus \{v\} \cup \{u\}$ is a facet of~$\C$ and $uv$ is not an edge in~$\C$. As $\facetA \in R_a(v)$ we have that there exists a $\smallface \in \U_b$ such that $\smallface\cap \facetA = \{v\}$. Since $uv$ is not an edge and $\smallface$ is a face containing~$v$, we have that $u\not\in \smallface$. Hence,  $\facetA' \cap \smallface = \emptyset$, in particular $\facetA'\cap \vertices(\lkcap{\C}{a}{b}) \notin \U_b^*=\U_a$, and $\facetA' \notin \family$.

Moreover, the only elements of~$\family$ that $\facetA'$ may not intersect are those in $R_b(v)$. To see this, assume that there exists a facet $\facetB \in \family$ such that $\facetA' \cap \facetB = \emptyset$. $\facetB \in \family_a$ is not possible since otherwise $\{a \} \subseteq \facetA' \cap \facetB \neq \emptyset$. If $\facetB \in \family_c$ for some $c \in \face \setminus \{a,b\}$, then, by assumption, $v\neq c$, $vc \not \in \C$ and thus $v \not \in \facetB$. It follows that $\facetB \cap \facetA \setminus \{v\} \neq \emptyset$ and thus $\facetB \cap \facetA' \neq \emptyset$, a contradiction. Hence, $\facetB \in \family_b$ which implies $\facetB \cap \facetA = \{ v\}$ and thus $\facetB \in R_b(v)$.

For every $\facetA\in R_a(v)$ choose any such $\facetA'$ and let $R'_a(v) = \set{\facetA' }{ \facetA \in R_a(v)}$. Notice that if $\facetA \neq \facetB$ then $\facetA' \neq \facetB'$. Otherwise assume $\facetA\neq \facetB$ but $\facetA'=\facetB'$, or, in other words, $\facetA' = \facetA\setminus \{v\} \cup \{u\} = \facetB \setminus \{v\} \cup \{w\} = \facetB'$, for $u \neq w$. In particular, $uv \in \facetB$ and $vw \in \facetA$, which is not possible since $uv, vw \not \in \C$.
Let $\family' = \family \setminus R_b(v) \cup R'_a(v)$. As $|R'_a(v)| = |R_a(v)|\ge|R_b(v)|$, $|\family'| \geq |\family|$. By the observations above $\family'$ is an intersecting family.

\medskip

In order to define $\U'$ consider the upper set $\U_b \setminus Q_b(v)$. By construction, $\family'_{b\bar a} |_{\lkcap{\C}{a}{b}} \subseteq \U_b \setminus Q_b(v)$. All minimal elements $\smallface$ of $\U_b$ having~$v$ as an element are not in $\U_b \setminus Q_b(v)$. Thus the minimal elements of $\U_b \setminus Q_b(v)$ are exactly the minimal elements of $\U_b$ that do not contain~$v$, and $E(\U_b \setminus Q_b(v)) \subseteq E(\U_b) \setminus \{v\}$. Now let
\[
\U' = \left( (\U_b \setminus Q_b(v))^*, (\U_b \setminus Q_b(v))^{**}  \right)
\]
By \Cref{prop:dual_pairs} we have that $\U'$ supports $\family'$ and $E(\U') = E((\U_b \setminus Q_b(v))^*) \subseteq E(\U_b \setminus Q_b(v)) \subseteq E(\U) \setminus \{v\}$.
\end{proof}

\begin{theorem}
\label{theorem_2_to_1_meep}
  Let~$\C$ be a complex with the missing edge exchange property and let~$\family$ be an intersecting family with base face $\face = ab$. Then there exists an intersecting family $\family'$ with $|\family'|\geq |\family|$ and base face either~$a$ or~$b$.
\end{theorem}

\begin{proof}
  Since $|\face|=2$, we can iteratively apply \Cref{lemma_vertex_flip} to obtain larger families supported in dual pairs with fewer essential vertices. This process necessarily terminates with an intersecting family supported in a dual pair without essential vertices. In particular, one of $\U_a$ and $\U_b$ must be empty and the other one must be equal to $\link_\C({a,b})$. If, w.l.o.g., $\U_a$ is empty, then $\family_{a\overline b}$ is empty, and $b$ on its own is a base face.
\end{proof}

\begin{corollary}
\label{corollary_2_to_1}
  Let~$\C^{(r)}$ be the $r$-skeleton of a flag pure complex without boundary $\C$ and let~$\family$ be an intersecting family with base face $\face = ab$. Then there exists an intersecting family $\family'$ with $|\family'|\geq |\family|$ and base face either~$a$ or~$b$.
\end{corollary}

\begin{proof}
  This is a direct consequence of \Cref{theorem_2_to_1_meep} combined with \Cref{prop:meep,prop:meep2}.
\end{proof}

Theorem~\ref{thm:main2} is the case $r=\dim(\C)$ (that is, $\C=\C^{(r)}$) of \Cref{corollary_2_to_1}.

\subsection{The $2$-dimensional case} We are now in the position to give a short proof of \Cref{conj:no_free_ridges} in dimension two. 

\begin{theorem}
  \label{thm:dim2}
  Let~$\C$ be a flag pure~$2$-complex without boundary edges.
  Then~$\C$ is pure-EKR.
\end{theorem}

 \begin{proof}
  Let~$\family$ be an intersecting family of facets of~$\C$ and $\face$ a minimal base face of~$\family$.
  We aim to show that there exists an intersecting family~$\family'$ with $|\family'| \geq |\family|$ together with a base face that is a vertex. Thus, assume that $\sigma$ is not a vertex.

  If $\face$ is an edge we can apply \Cref{corollary_2_to_1}. Hence, assume $\face$ is a triangle, say $\face = abc$.

  By minimality of $\face$ as a base face, there exist triangles $\smallface_a, \smallface_b, \smallface_c \in \family$ with $\face\cap \smallface_a = a$, $\face\cap \smallface_b = b$, and $\face\cap \smallface_c = c$.
  Since~$\C$ is flag, $\smallface_a$, $\smallface_b$ and $\smallface_c$ cannot intersect in a common vertex. Since they must intersect pairwise, we conclude that $\smallface_a = auv$, $\smallface_b = bvw$, and $\smallface_c = cuw$ for vertices $u$,~$v$, and $w$.
  $\face$ may or may not belong to~$\family$, but we claim that~$\family$ is contained in $\{ \face, \smallface_a, \smallface_b, \smallface_c \}$. Indeed, if this is not the case then any additional triangle $\smallface \in \family$ must intersect $\face$ in an edge or in a vertex since $\sigma$ is a base face, but both things are impossible:
  \begin{itemize}
  \item If $\smallface$ intersects $\face$ in an edge, say $ab$, then the remaining vertex of $\smallface$ must be in $\smallface_c$ (since $\smallface$ and $\smallface_c$ meet). Hence, assume $\smallface=abu$. Then $buvw$ is a $4$-clique and thus contradicts the flagness of~$\C$.
  \item If $\smallface$ intersects $\face$ in a vertex, say~$a$, then for $\smallface$ to intersect both $\smallface_b$ and $\smallface_c$
  either $w\in \smallface$ (producing the $4$-clique $auvw$) or $\smallface = auv =\smallface_a$.
  \end{itemize}
  Altogether, $|\family|\leq 4$ and the statement follows by the fact that the star of any vertex in a flag~$2$-complex without boundary must be of size at least four.
\end{proof}

\section{The $3$-dimensional case}
\label{sec:three}

The following is a slightly more precise version of~\Cref{thm:main}. 

\begin{theorem}
\label{theorem:3dim}
  Let~$\C$ be a flag pure $3$-complex without boundary triangles.
  Let~$\family$ be an intersecting family of facets and let $\face\in \C$ be a base face of minimal size for~$\family$.
  Then, either $|\face|=1$ or there exists an intersecting family $\family'$ with $|\family'| \ge |\family|$ and with a base face $\face'$ properly contained in $\face$.
\end{theorem}

To see that~\Cref{theorem:3dim} implies~\Cref{thm:main} observe that 
given a flag pure $3$-complex~$\C$ without boundary triangles, and given an intersecting family~$\family$ of~$\C$ with base face $\face$, we can iteratively apply \Cref{theorem:3dim} until we arrive at a family $\family'$ contained in a vertex star of~$\C$, such that $|\family'|  \geq |  \family|$.

The proof of~\Cref{theorem:3dim} considers the three possibilities for the size of $\face$ separately. Size two is settled by \Cref{corollary_2_to_1} and sizes three and four are handled in \Cref{lemma:3to2,lemma:4to3}, in~\Cref{ssec:3to2,ssec:4to3} below. 
Before presenting their proofs we need to introduce some additional machinery which occupies the whole of \Cref{ssec:sectors,ssec:flag1cmplxs}. 

\subsection{The outerlink and its sectors}
\label{ssec:sectors}

Given a pure simplicial complex~$\C$ and a face $\face \in \C$, we denote by $\C/\face$ the contraction of $\face$ in~$\C$. In general this is not a simplicial complex, but rather a CW-complex  and also a quotient space of~$\C$ in the topological sense. It converts each individual facet in~$\C$ whose intersection $\face' $ with $\face$ is $\ell$-dimensional ($\ell\ge 0$)  into a $(k-\ell)$-simplex $(\face \setminus \face') \cup \{ v\}$, where~$v$ is a symbol denoting a new vertex of $\C/\face$ that replaces the whole face $\face$.
The glueings between simplices in $\C / \face$ are the ones coming from~$\C$, and faces not meeting $\face$ are unaffected.

\medskip
The reason why $\C/\face$ may not be a simplicial complex is that
 if $\face'$ and $\face''$ are two facets with $\face \setminus \face' = \face \setminus \face''$ then in $\C/\face$ they produce two different simplices with the same vertex set.
 (E.g., if in a simple graph we contract an edge which is part of a three cycle, the other two edges in the cycle are still two distinct edges now with the same end-points, and the contracted graph is not simple anymore).
 One advantage of flag complexes is that, for them, this cannot happen:

\begin{lemma}
  \label{lem:contract}
  A simplicial complex~$\C$ is flag if and only if $\C / \face$ is a simplicial complex for every face $\face$ of~$\C$.
\end{lemma}

\begin{proof}
  First, assume that $\C/\face$ is not simplicial.
  That is, there exist two faces~$\face_1$ and~$\face_2$ of $\C/\face$ with the same vertex set.
  Let $\tilde{\face}_1$ and $\tilde{\face}_2$ be faces in~$\C$ contracting to them.
  Let $\smallface_1 = \left(\tilde{\face}_1 \cap \face \right) \setminus \tilde{\face}_2$, $\smallface_2 = \left(\tilde{\face}_2 \cap \face \right) \setminus \tilde{\face}_1$ and $\smallface_3 = \tilde{\face}_1 \setminus \face  = \tilde{\face}_2 \setminus \face$.
  Then $\smallface_1 \cup \smallface_2 \cup \smallface_3 = \tilde{\face}_1 \cup \tilde{\face}_2$ is a clique but not a face in~$\C$, implying that~$\C$ is not flag.

  On the other hand, if there exists a minimal non-face of dimension at least~$2$, then contracting any of its proper faces, e.g. an edge, yields a CW-complex that is not simplicial.
\end{proof}

\begin{definition}
  \label{def:outerlink}
  Let $\face$ be a face in a flag complex~$\C$.
  The \defn{outerlink} of $\face$ in~$\C$ is given by all faces~$\face'$ of~$\C$ such that $\face\cap\face'=\emptyset$ and $\face'\cup v\in \C$ for some $v\in \face$.
\end{definition}

\begin{remark}
Observe that since $\C$ is flag, the outerlink of $\face$ is equal to the link of (the vertex corresponding to) $\face$ in the simplicial complex $\C/\face$.
  In fact, an alternative definition of $\C/\face$ for a flag complex~$\C$ is that it equals the union of the complement of $\face$ and the pyramid $S*\{v\}$, where $S$ is the outerlink of $\face$ in~$\C$ and $v$ is a new vertex.

  For a flag~$d$-manifold, the outerlink of every face~$\face$ is a $(d-1)$-sphere: It is the link of a vertex in the manifold $\C/\face$, which is homeomorphic to~$\C$.
  Even more, the complement $\C \setminus \face$ of $\face$ in~$\C$ is a~$d$-manifold with boundary, and its boundary is equal to the outerlink of $\face$ in~$\C$.

\end{remark}

\begin{definition}
  \label{def:sectors}
  Let~$\C$ be a flag pure~$d$-complex, let $\face \in \C$ be one of its faces, and let $S=\link_{\C/\face} (\face)$ be the outerlink of $\face$.
  For each vertex $a\in \face$ we define the \emph{sector} of~$a$ in $S$ to be the subcomplex of $S$ given by
  \[
    S_a :=\link_\C(a) \setminus \face.
  \]
\end{definition}

Throughout the rest of this section we use the notation~$S$ for the outerlink of a face $\face$,~$S_a$ for the sector corresponding to a vertex $a\in \face$, and $S_\smallface := \cap_{a\in \smallface} S_a$ for any intersection of sectors ($\smallface\subset \face$). If $|\smallface|$ equals two or three we call $S_\smallface$ a \defn{bisector} or \defn{trisector}, respectively.

\begin{lemma}
  \label{lem:induced}
Let~$\C$ be a flag pure $d$-complex, let $\face \in \C$ be one of its faces, and let $S=\link_{\C/\face} (\face)$ be the outerlink of $\face$. Then the following two statements hold.
\begin{enumerate}
\item Sectors and their intersections are flag induced subcomplexes of $\C$.

\item If $\C$ is without boundary then the intersection of any $k$ sectors is pure of dimension $d-k$. In particular, $S$ is pure of dimension $d-1$ and each facet of $S$ belongs to exactly one sector.
\end{enumerate}
\end{lemma}

\begin{proof}
Let $\smallface\subset \face$ and consider the intersection of sectors $S_\smallface$.

For part (1), assume $\C$ is flag and let us show that $S_\smallface$ is induced, which implies it is flag. If a face $\rho \in \C$ has all vertices in $S_{\smallface}$ then $\rho \cup \smallface$ is a clique in $\C$: $\rho$ and $\tau$ are faces, hence cliques, and for every pair of vertices $a\in\tau$, $b \in \rho$, $ab$ spans an edge since $b$ is a vertex in $S_a$. By the flagness of $\C$ we then have $\rho \cup \smallface \in \C$ and, thus, $\rho \in S_{\smallface}$.

  For part (2) let $k=|\smallface|$ and assume $\C$ is without boundary. Equivalently, by \Cref{prop:meep}, it has the missing edge exchange property.
  The intersection $S_\smallface$ is at most $(d-k)$-dimensional because it is part of the link of~$\smallface$. 

  Now let $\rho$ be a facet of $S_{\smallface}$. 
	By part (1) $\rho\cup \smallface \in \C$ and so it is contained in a facet $\facetA$ of $\C$. 
	By maximality of $\rho$, $\facetA\setminus \face = \rho$. If there exists $v\in \facetA\cap(\face\setminus\smallface)$ then, by the missing edge exchange property, there exists $u$ such that $\facetB = (\facetA\setminus v)\cup u$ is a facet and
$uv\notin \C$; in particular $u\notin\face$. 
	But $\facetB \supseteq \rho\cup\smallface\cup u$, so $\rho\cup u \in S_{\smallface}$ which is a contradiction to the maximality of $\rho$. 
	Then $\facetA = \rho\cup\smallface$ which means $|\rho| = d-k+1$, and $S_{\smallface}$ is pure of dimension $d-k$. 
	In particular $S$ is pure of dimension $d-1$ with its facets partitioned into sectors.	
 \end{proof}

\subsection{A classification of dual pairs in flag 1-complexes}
\label{ssec:flag1cmplxs}

In this section we establish a classification of dual pairs of flag $1$-dimensional complexes which forms the base of the proof of \Cref{theorem:3dim}. Namely, the classification allows us to perform an induction on the number of essential vertices, as in \Cref{corollary_2_to_1}, separately for base faces of sizes three and four.

\begin{proposition}
\label{lem:1dflag}
  Let~$\D$ be a flag $1$-dimensional complex.
  Then every dual pair in~$\D$ is isomorphic to one of the following types:
  \[
    (\emptyset, \D),\qquad (\langle v\rangle, \langle v\rangle), \qquad
    (\langle u,v\rangle, \langle uv\rangle) \qquad (\langle uv,wx\rangle,\langle vw,ux\rangle).
  \]
\end{proposition}

\begin{proof}
Let $\U = (\U_1,\U_2)$ be a dual pair. If one of $\U_1$ or $\U_2$ is empty, then $\U$ is isomorphic to $(\emptyset, \D)$.

Suppose that $\U_1$ contains a singleton~$v$. If $\U_1 = \langle v \rangle$ or~$v$ is also a singleton in $\U_2$, then $\U_1 = \U_2 = \langle v \rangle$. Hence, assume that~$v$ is not a singleton in $\U_2$. This implies that there exists an element $\smallface \in \U_1$ not containing~$v$, and a generating edge $uv \in \U_2$. Since $\U_1 = \U_2^*$, $\smallface \cap uv \neq \emptyset$ and thus $\smallface \cap uv = u$ which implies $u \in \U_1$. Since $(\langle u,v\rangle,\langle uv \rangle)$ is a dual pair, and, by \Cref{prop:dual_pairs}, dual pairs cannot be properly contained in larger dual pairs, we are done.

It remains to assume that neither $\U_1$ nor $\U_2$ contains a singleton. Let $uv$ be an edge in $\U_1$. Then $\U_2$ must contain edges $ux$ and $vw$ for some vertices $x$ and $w$ with $\{x,w\} \cap \{u,v\} = \emptyset$.
Since~$\D$ does not have $3$-cycles, $uw$ and $vx$ are not in~$\D$ and $x \neq w$. Since $\U_1$ is not generated by $uv$ alone (otherwise $\U_2$ contains singletons), there must be another edge in $\U_1$ both intersecting $ux$ and $vw$. It thus follows that this edge must be equal to $xw$. Again, since $(\langle uv,wx\rangle,\langle vw,ux\rangle)$ is a dual pair we are done by \Cref{prop:dual_pairs}.
\end{proof}

Note that every dual pair in a flag $1$-dimensional complex is isomorphic to a dual pair in the boundary of the~$2$-dimensional cross polytope (that is, the $4$-cycle), see \Cref{ex:dual1}. This motivates the following more general conjecture:

\begin{conjecture}
\label{conj:crosspolytopes}
  Let $\U=(\U_1,\U_2)$ be a dual pair in a~$d$-dimensional flag simplicial complex without boundary~$\D$. Then, $\U$ is isomorphic to a dual pair in the boundary complex of the $(d+1)$-dimensional cross-polytope.
\end{conjecture}

\begin{corollary}[of \Cref{lem:1dflag}]
\label{lemma:intersection_cases}
Let~$\C$ be a flag $3$-complex without boundary triangles,~$\family$ an intersecting family of tetrahedra of~$\C$, $\face \in \C$ a minimal base face of~$\family$ containing at least two vertices $a,b\in\C$, and let $\U = (\U_a,\U_b)$ be a dual pair supporting~$\family$ at $ab$. Then, modulo exchange of~$a$ and $b$, one of the following three situations occur.
\begin{enumerate}[I)]
\item $\U_a = \U_b = \langle v \rangle$ for some vertex $v \in \link_\C (ab)$.
In particular, all facets of $\family_{a\overline b}$ and $\family_{b \overline a}$ have a common vertex in $\link_\C ({ab})$.

\item $\U_a = \langle u,v \rangle$, $\U_b = \langle uv \rangle$ for an edge $uv \in \link_\C (ab)$. In particular, all facets of $\family_{a \overline b}$ contain the edge $uv$.

\item  $\U_a = \langle uv, wx \rangle$, $\U_b = \langle vw, ux \rangle$, for some induced cycle $u, v, w, x$ in $\link_\C (ab)$ in that order. In particular,
\begin{itemize}
\item all facets in $\family_{a \overline b}$ contain one of $uv$ or  $wx$, and
\item all facets in $\family_{b \overline a}$ contain one of $vw$ or  $ux$.
\end{itemize}
\end{enumerate}
\end{corollary}

\begin{proof}
  Note that, since~$\C$ is flag and $ab$ is an edge of~$\C$, the link of $ab$ in~$\C$ equals 
  $\lkcap{\C}{a}{b}$ and it is hence an induced subcomplex. In particular, it is 
  a flag $1$-dimensional complex. Moreover, note that $\face$ is minimal and thus $(\U_a,\U_b)$ cannot be isomorphic to $(\emptyset, \link_\C (ab))$. \Cref{lem:1dflag} then implies that we only have the three cases in the statement.
\end{proof}

We say that $\family_{a\overline b}$ and $\family_{b \overline a}$ have intersection type I, II or III according to the above cases.
We can now prove the remaining cases of \Cref{theorem:3dim}, namely, base faces of size three (\Cref{lemma:3to2}) in \Cref{ssec:3to2} and size four (\Cref{lemma:4to3}) in \Cref{ssec:4to3}.

\subsection{Base faces of size three}
\label{ssec:3to2}

In the proof we remove and replace facets of an intersecting family~$\family$ with base face $\face$ in order to construct a new intersecting family $\family'$ with $|\family'| \geq |\family|$ and a base face $\face'$ with $|\face| > |\face'|$.
This is only possible if we can find facets not in~$\family$ that intersect both $\face'$ and all facets in~$\family$.

The following basic observation is a very useful tool to find such facets.
We use it several times, sometimes implicitly, which is why we state it here.

\begin{lemma}
  \label{lem:opp}
  Let~$\C$ be a flag pure $d$-complex, let $\family \subset \C$ be an intersecting family of facets, and let $\facetA \in \family$. Let $\facetB$ and $\facetC$ be facets in~$\C$ such that $\facetB\cap\facetA$ and  $\facetC\cap\facetA$ partition $\facetA$. Then at most one of $\facetB$ and $\facetC$ is in $\family$.
\end{lemma}

\begin{proof}
  Assume otherwise that there exist facets $\facetB, \facetC\in \family$ such that $\facetB\cap\facetA$ and  $\facetC\cap\facetA$ partition $\facetA$.
  Since~$\family$ is an intersecting family, there exists a vertex $v \in \facetB\cap\facetC$. The hypothesis on $\facetB$ and $\facetC$ imply that $v\notin\facetA$.
  But then $\facetA \cup v$ spans a complete graph of $d+2$ vertices, a contradiction to~$\C$ being flag of dimension~$d$.
\end{proof}

Given an intersecting family~$\family$ with base face $\face$, and $\smallface \subset \face$, we define 
\[
\family_\smallface^\circ := \set{\facetA \in \family }{ \facetA\cap \face = \smallface}.
\]

\begin{lemma}
\label{lemma:3to2}
  \Cref{theorem:3dim} holds in the case that $|\face| = 3$.
\end{lemma}

\begin{proof}
Let $\face = abc$, and let~$S$ be the outerlink of $abc$. Note that the trisector $S_{abc}$ equals the link of $abc$ in~$\C$, hence it is
a collection of vertices (at least two since otherwise $abc$ is a boundary triangle in~$\C$) spanning no edge.
See \Cref{lem:induced} for further properties of the outerlink and its sectors. We have three cases:

\begin{itemize}
  \item \emph{$\family_{a\overline b}$ and $\family_{b\overline a}$ have intersection type III. } If $c$ is not part of the $4$-cycle, then there must be at least two opposite vertices of the $4$-cycle not in $S_{abc}$. If $c$ is part of the $4$-cycle, the vertex opposite $c$ in the $4$-cycle is not in $S_{abc}$. Hence, in both cases we can use \Cref{lemma_vertex_flip} to obtain a family $\family'$ with $|\family'| \geq |\family| $ and less essential vertices in $(\U_a,\U_b)$. In particular, now $\family'_{a\overline b}$ and $\family'_{b\overline a}$ have intersection type I or II.

  \item \emph{$\family_{a\overline b}$ and $\family_{b\overline a}$ have intersection type II. } W.l.o.g., say that $\U_a = \langle uv \rangle$, and $\U_b = \langle u,v \rangle$ for two vertices $u,v \in \link_\C (ab)$. If $c \not \in uv$, then one of $u$ and~$v$ is not in $S_{abc}$ and, again, we can apply \Cref{lemma_vertex_flip} to obtain intersection type I. If $c \in uv$, then for all $\facetA \in \family_{a\overline b}$ we have $c \in \facetA$. It follows that $\family_a^\circ = \emptyset$, and $\face' = bc$ is a smaller base face.

  \item \emph{$\family_{a\overline b}$ and $\family_{b\overline a}$ have intersection type I. } Say $\U_a = \langle v \rangle = \U_b$. If $v=c$, then either $ac$ or $bc$ is a base face. If $v \neq c$ and $v \not \in S_{abc}$ then we can apply \Cref{lemma_vertex_flip}. Hence, assume $v \in S_{abc}$. As $\face$ is minimal in size, $vc$ is not a base face. This implies that there exists a facet $\facetA \in \family_{ab}$ such that $v\notin \facetA$. But $\facetA$ must intersect elements in $\family_c^\circ$ and hence there must exist a vertex $u\in \facetA$ such that $u \in S_{abc}$. As $u$ and~$v$ are in the same trisector, $uv \notin \C$. In particular $u\notin \facetB$ for any $\facetB \in \family_a^\circ \cup \family_b^\circ$. Let $\facetC \in \family_c^\circ$. Then $\facetB\cap \vertices(S_{ac}) = uw$ and $w\in \facetB$ for any $\facetB \in \family_a^\circ$. Similarly, $\facetB\cap \vertices(S_{bc}) = ux$ and $x\in \facetB$ for any $\facetB \in \family_b^\circ$. It follows that the only possible element in $\family_c^\circ$ is $cuwx$.

Let $y$ be any neighbor of~$v$ in the $S_{ab}$. In particular, $y \neq u$. Then $abvy \not \in \family$ because it does not intersect with $cuwx$. Setting $\family' = \family \setminus \{ cuwc \} \cup \{ abvy \}$ yields an intersecting family of equal size with base face $ab$.
\end{itemize}

\end{proof}

\begin{remark}
\label{rem:counter2}
  Without the assumption that $\face$ has minimal size among base faces for~$\family$ \Cref{lemma:3to2}, and hence~\Cref{theorem:3dim}, is not true, as the following example shows:

  \centerline{\includegraphics[height=6cm]{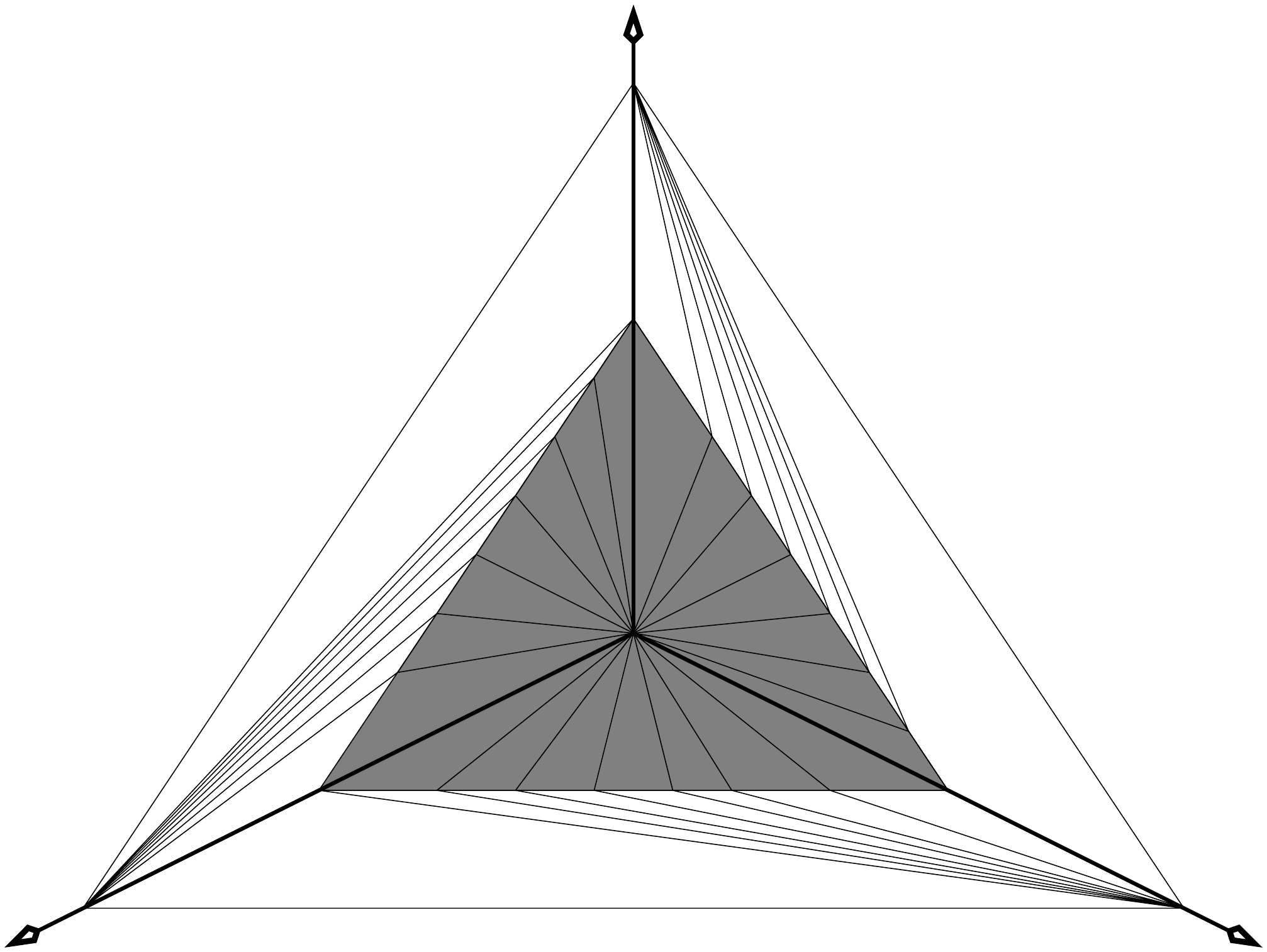}}

  The figure shows the outerlink of a triangle $abc$ and the tetrahedra corresponding to the shaded triangles form an intersecting family~$\family$ of cardinality $3k$ (where $k=7$ in the figure), containing the central point~$v$ which is in the trisector.
  Thick rays are bisectors, leading to a second point in the trisector.
  The triangle $abc$ is a base face for~$\family$, but every intersecting family with base face properly contained in $abc$ has size roughly $2k$.
  However, $v$ is a smaller base face for~$\family$.
\end{remark}

\subsection{Base faces of size four}
\label{ssec:4to3}

\begin{lemma}
\label{lemma:4to3}
  \Cref{theorem:3dim} holds in the case that $|\face| = 4$.
\end{lemma}

\begin{proof}
  Let $\face= abcd$ be a base face of minimal size.
  In particular, no triangle of~$\C$ intersects all facets of~$\family$, and all of $\family_a^\circ$, $\family_b^\circ$, $\family_c^\circ$ and $\family_d^\circ$ are non-empty.
  Moreover, we can assume that each of them contains at least two facets.
  To see this, w.l.o.g. assume that $\family_a^{\circ} = \{ auvw \}$.
  Because~$\C$ has no boundary triangles, there must exist a facet $bcdx \in C$.
  Since~$\C$ is flag, $x \not \in \{u,v,w \}$ and, by~\Cref{lem:opp}, $bcdx \not \in \family$.
  It follows that $\family' = \family \setminus \{ auvw \} \cup \{ bcdx \}$ is an intersecting family with base face $bcd$ and $|\family'| = | \family |$.

We have three cases, namely
\begin{enumerate}
  \item \label{dim3:typeI} Some pair, say $\family_{a\overline b}$ and $\family_{b \overline a}$,  have intersection of type I;

  \item \label{dim3:typeII} No pair has intersection type I and some pair, say $\family_{a\overline b}$ and $\family_{b \overline a}$, have intersection of type II; or

  \item \label{dim3:typeIII} All pairs $\family_{x\overline y}$ and $\family_{y \overline x}$, $x,y \in \face$ have intersection of type III.
\end{enumerate}

We treat the three cases individually:
\begin{enumerate}
  \item[\eqref{dim3:typeI}]  Let $\U_a = \langle v \rangle = \U_b$. If $v \in \{ c,d\}$, then either $acd$ or $bcd$ is a base face. If $v \not \in \{ c,d\}$ and~$v$ is not in any trisector, then we can apply \Cref{lemma_vertex_flip}. Hence, assume~$v$ is in a trisector, say $S_{abc}$. As $\face$ is minimal in size, $vcd$ is not a base face. This implies that there exists a facet $\facetA \in \family_{ab}$ such that $v\notin \facetA$. But $\facetA$ must intersect elements in $\family_c^\circ$ and hence there must exist a vertex $u\in \facetA$ such that $u \in S_{abc}$. As $u$ and~$v$ are in the same trisector, $uv \notin \C$. In particular, for every $\facetB \in \family_a^\circ \cup \family_b^\circ$ we have that $u\notin \facetB$. Let $\facetC \in \family_c^\circ$. Then $\facetB\cap \vertices(S_{ac}) = uw$ and $w\in \facetB$ for any $\facetB \in \family_a^\circ$. Similarly, $\facetB\cap \vertices(S_{bc}) = ux$ and $x\in \facetB$ for any $\facetB \in \family_b^\circ$. It follows that the only possible element in $\family_c^\circ$ is $cuwx$. But we assumed that $| \family_c^\circ | \geq 2$.

 \item[\eqref{dim3:typeII}]  Let $\U_a = \langle uv \rangle$ and $\U_b = \langle u,v \rangle$ for two vertices $u,v \in \link_\C (ab)$. If $c \in uv$ ($d \in uv$), then for all $\facetA \in \family_{a\overline b}$ we have $c \in \facetA$ ($d \in \facetA$). It follows that $\family_a^\circ = \emptyset$ and $bcd$ is a smaller base face, contradiction. Moreover, if one of $u$ and~$v$ is not in any trisector, we can apply \Cref{lemma_vertex_flip} to obtain intersection type I.

Hence, assume both $u$ and~$v$ are in trisectors. In particular, they must be in distinct trisectors, say, $u \in S_{abc}$ and $v \in S_{abd}$.

All elements of $\family_{c \overline a}$ must contain $u$. To see this, suppose $\facetA \in \family_{c \overline a}$ and $u \not \in \facetA$ (hence $a,u \not \in \facetA$). By assumption, there are at least two tetrahedra in $\family_a^\circ$, $auvx$ and $auvy$. Moreover, if $v \in \facetA$, then  $\{a,b,c,d,v\}$ spans a $5$-clique (recall that $v \in S_{abd}$). Hence $a,u,v \not \in facetA$ and thus $\facetA$ must contain both $x$ and $y$. But this implies that $\{a,u,v,x,y\}$ spans a $5$-clique in~$\C$, a contradiction.

Furthermore, all elements of $\family_{a \overline c}$ must contain $u$: Suppose otherwise that there exists a facet $\facetA \in \family_{a \overline c}$, $u \not \in \facetA$. Then, necessarily $b \in \facetA$ (otherwise $\facetA \in \family_{a \overline b}$, contradiction to $\U_a = \langle uv \rangle$). We know that for all $\facetB \in \family_{c \overline a}$ we have that $u \in \facetB$. Furthermore, let $x \in \facetA \cap \facetB$. Altogether, we have that $\{a,b,x\}\subset \facetA$, and $\{ c,u,x\} \subset \facetB$. But then~$\C$ must contain a $5$-clique spanned by $\{a,b,c,u,x\}$. Contradiction to the flagness of~$\C$.

But then all elements of $\family_{a \overline c}$ and $\family_{c \overline a}$ contain $u$. In other words, the dual pair $(\langle u \rangle,\langle u \rangle)$ in $\link_\C (ac)$ supports~$\family$ at $ac$, which means~$a$ and $c$ have intersection type I.

  \item[\eqref{dim3:typeIII}] Let $\U_a = \langle st, uv \rangle$, and $\U_b = \langle tu, sv \rangle$ for vertices $s,t,u,v \in \link_\C (ab)$ forming a $4$-cycle in that order.

If both $c,d \in \{s,t,u,v \}$ then, by flagness, $cd$ must be an edge of the $4$-cycle. W.l.o.g., let $st = cd$. Then every element of $\family_{b \overline a}$ must contain either~$d$ or $c$. This implies $\family_b^\circ$ is empty, a contradiction to the assumption that $\family_b^\circ$ is of size at least two.

Hence, w.l.o.g., assume that $d \not \in \{s,t,u,v \}$ and $s=c$. We want to show that, in this case, $acu$ is a base face of~$\family$ -- a contradiction to the assumption that $\face$ is minimal. All elements of $\family_b$ either contain~$a$, or, in case they belong to $\family_{b \overline a}$, contain $u$ or $c$ by the way $\family_{a\overline b}$ and $\family_{b \overline a}$ intersect. To see that all elements of $\family_d$ either contain~$a$, or $u$, note that, by assumption, there exist two facets $auvx, auvy \in \family_a^{\circ}$ (since $s=c$, none of the facets of $\family_{a \overline b}$ containing $st$ are in $\family_a^{\circ}$). No facet $\facetA \in \family_d$ can contain~$v$, since otherwise $\{a,b,c,d,v\}$ spans a $5$-clique (note that $av$, $bv$ and $cv$ are all edges). Hence, a facet $\facetA \in \family_{d\overline a}$ with $u\not \in \facetA$ must intersect $auvx$ in $x$ and $auvy$ in $y$. In particular, $xy$ is an edge and $\{a,u,v,x,y\}$ spans a $5$-clique, a contradiction. Hence, $u \in A$ for all $A \in \family_{d\overline a}$ and $acu$ is a base face.

Hence, we can assume that $c,d \not \in \{s,t,u,v\}$, and by symmetry, that no $4$-cycle corresponding to the dual pair of a vertex pair $x,y \in \face$ contains a vertex of $\face$. In other words, all $4$-cycles are in the outerlink~$S$ of $\face$.

\medskip

First note that, in this case, $\family_{xy}$, $x,y \in \face$, must be empty. Otherwise assume that, w.l.o.g., there exists a facet $acst \in \family$, $s,t \not \in \{b,d\}$. But $s \in S_{abc}$ and $t \in S_{abd}$ by assumption and thus facets $abcs$ and $abdt$ belong to~$\C$, and $\{ a,b,c,s,t \}$ forms a $5$-clique, a contradiction to the flagness of~$\C$. By symmetry, the same applies to all $\family_{xy}$, $x,y \in \face$.

Moreover, $\family_{xyz}$, $x,y,z \in \face$, must be empty as well. Otherwise, w.l.o.g., let $abcs \in \family$, $s \neq d$. Since $\family_d^{\circ}$ is not empty it follows that $\{a,b,c,d,s\}$ spans a $5$-clique, a contradiction.

Altogether it thus follows that $\family = \family_a^{\circ} \cup \family_b^{\circ} \cup \family_c^{\circ} \cup \family_d^{\circ} \cup \{ \face \}$. Moreover, every element of, say, $\family_a^\circ = \family_{a\overline b}= \family_{a\overline c}= \family_{a\overline d}$ intersects the outerlink in a triangle with one edge each in bisectors $S_{ab}$, $S_{ac}$ and $S_{ad}$, and one vertex each in trisectors $S_{acd}$, $S_{abd}$ and $S_{abc}$. Also, the intersection of a bisector, say $S_{ab}$, with~$\family$ must be a $4$-cycle with pairs of non-adjacent vertices in trisectors $S_{abc}$ and $S_{abd}$ respectively.

\medskip

Let $stw$ and $uvx$ be triangles of $S_a$ corresponding to elements in $\family_a^\circ$ and let $tuy$ and $svz$ be triangles of $S_b$ corresponding to elements in $\family_b^\circ$ (if one of these triangles does not exist,~$a$ and $b$ must have intersection type II or I). Now, either $sw$ or $tw$ must be the edge of the $4$-cycle in $S_{ac}$ corresponding to the intersection between between $\family_a$ and $\family_c$. We present the proof of case \eqref{dim3:typeIII} for $sw$ being in the $4$-cycle of $S_{ac}$. The argument for $tw$ in the $4$-cycle of $S_{ac}$ is completely analogous due to symmetry.

It follows that $s$ and $u$ must be part of trisector $S_{abc}$, $t$ and~$v$ must be part of trisector $S_{abd}$, and $w$ and $x$ must be part of trisector $S_{acd}$. Moreover, it follows that the $4$-cycle of $S_{ac}$ must be $(s,w,u,x)$ in that order, since $w$ and $x$, being part of the same trisector, cannot span an edge.

Now every element of $\family_c^\circ$ containing edge $uw$ must contain $z$ to intersect $svz$ -- and thus the only option is $cuwz$. Similarly the only element of $\family_c^\circ$ containing edge $sx$ is $csxy$. Furthermore, for elements of $\family_d^\circ$ to intersect with all other elements of~$\family$ in type III, they must intersect the outerlink triangles $txz$ and $vwy$.

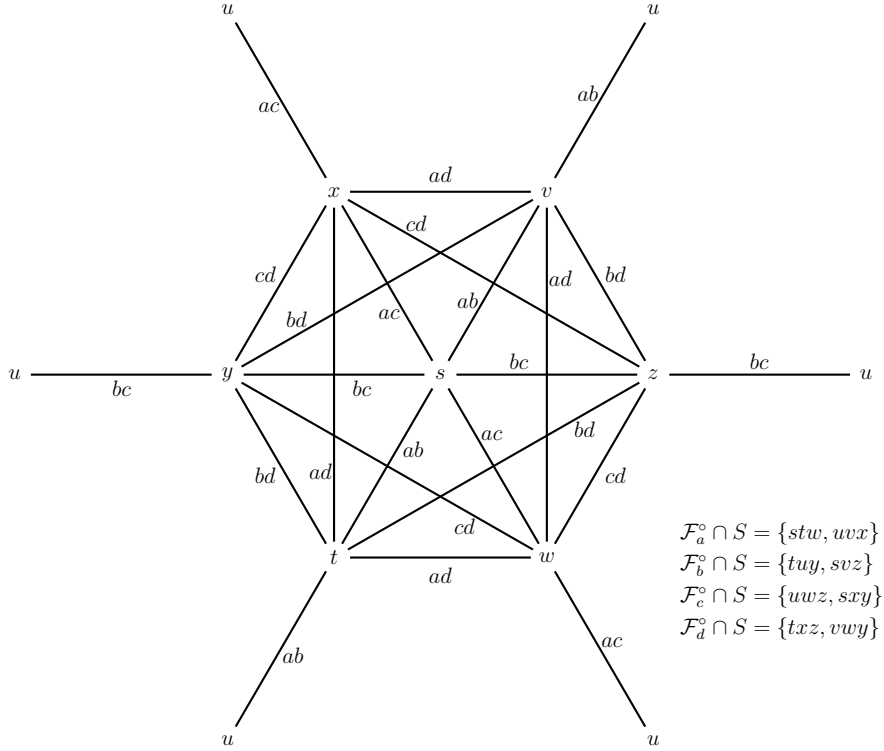
\begin{figure}
  \begin{center}
    \begin{tikzpicture}[scale=.7, every node/.style={scale=0.7}]
      \ngon{6}{(0,0)}{4}
          {1/2/black,2/3/black,3/4/black,4/5/black,5/6/black,6/1/black,1/3/black,2/4/black,3/5/black,4/6/black,5/1/black,6/2/black}
        {6}{$w$,$t$,$y$,$x$,$v$,$z$}{first}{0}

      \ngon{6}{(0,0)}{8}
          {}{6}{$u$,$u$,$u$,$u$,$u$,$u$}{second}{0}

      \ngon{6}{(0,0)}{6}
          {}{6}{$ac$,$ab$,$bc$,$ac$,$ab$,$bc$}{third}{2.5}

      \ngon{6}{(0,0)}{1.5}
          {}{6}{$ac$,$ab$,$bc$,$ac$,$ab$,$bc$}{third}{10}

      \ngon{6}{(0,0)}{3.8}
          {}{6}{$cd$,$ad$,$bd$,$cd$,$ad$,$bd$}{third}{30}

      \ngon{6}{(0,0)}{2.9}
          {}{6}{$cd$,$ad$,$bd$,$cd$,$ad$,$bd$}{third}{-21}

      \foreach \x in {1,...,6} {
        \draw[thick,black,shorten <=6pt, shorten >=6pt] {(first\x)--(second\x)};
      }

      \foreach \x in {1,...,6} {
        \draw[thick,black,shorten <=6pt, shorten >=6pt] {(first\x)--(0,0)};
      }

      \node[inner sep=0pt] at (0,0) {$s$};

      \node[inner sep=0pt, right] at (4.5,-3) {$\family_a^{\circ} \cap S = \{stw, uvx \}$ };
      \node[inner sep=0pt, right] at (4.5,-3.6) {$\family_b^{\circ} \cap S = \{tuy, svz \}$ };
      \node[inner sep=0pt, right] at (4.5,-4.2) {$\family_c^{\circ} \cap S = \{uwz, sxy \}$ };
      \node[inner sep=0pt, right] at (4.5,-4.8) {$\family_d^{\circ} \cap S = \{txz, vwy \}$ };

    \end{tikzpicture}
  \end{center}
  \caption{All $24$ edges of~$\family$ in the outerlink $S$ together with the bisectors they are contained in. Each bisector meets~$\family$ in a $4$-cycle, as can be read off the edge labels. Trisectors can be determined by the union of the bisectors of incident edges. Note that the deletion of $\{s,u\}$, $\{t,v\}$, $\{w,x\}$, or $\{y,z\}$ yields the set of edges with labels containing $d$, $c$, $b$, or~$a$ respectively. In each case it is the graph of an octahedron. \label{fig:typeIII}}
\end{figure}

At this point the intersection of~$\family$ with the outerlink~$S$ consists of eight triangles, with their edges forming six edge disjoint $4$-cycles, each contained in one of the six bisectors. Every additional element of~$\family$ must intersect $S$ in another triangle without adding any new edges to the intersection. But every such additional triangle must {\em (a)} intersect the three matching bisectors in one edge each, and {\em (b)} these edges must each be one of the two edges fixed by the respective intersection pattern of type III. For each vertex $x \in  \face$, the twelve edges of type $xy_i$, $y_i \in \face$, form the graph of an octahedron. That is, for each $x \in \face$ there are eight triangles satisfying {\em (a)} for the intersection of an element of $\family_x^{\circ}$ with $S$. But only two of these eight triangles satisfy condition {\em (b)}. See also \Cref{fig:typeIII} for details. However, all of the resulting $2 \times 4 = 8$ such triangles of $S$ already correspond to intersections of elements of~$\family$ with $S$. In particular, there cannot be any additional element in~$\family$ and thus~$\family$ must be equal to
\[
 abcd, astw, auvx, bsvz, btuy, csxy, cuwz, dtxz, dvwy 
\]
of size nine. Moreover, there is only one choice involved in the construction of~$\family$ with outcomes leading to isomorphic families. Hence, up to isomorphism, this is the unique intersecting family of tetrahedra such that all pairs $\family_{x\overline y}$ and $\family_{y \overline x}$, $x,y \in \face$ have intersection of type III and are disjoint of the base face.

The edges of the elements of~$\family$ form the complete $4$-partite graph with vertex partitions $\{a,y,z\}$, $\{b,w,x\}$, $\{c,t,v\}$, and $\{d,s,u\}$. Its (flag) clique complex~$\D$ is pure of dimension three, every ridge is in three facets, and every vertex star is of size $27$. Since~$\family$ is unique and~$\D$ is the smallest complex containing~$\family$ (and, thus, in particular $\D \subset \C$), this case cannot lead to a counterexample.
\qedhere
\end{enumerate}
\end{proof}

Note that the complex~$\D$ in the proof of \Cref{lemma:4to3}\eqref{dim3:typeIII} is the clique complex of a $4$-partite graph with every part being of size exactly three.
In particular,~$\D$ satisfies the assumption of Holroyd, Spencer and Talbot's \Cref{thm:HST05}, implying immediately that~$\D$ is pure-EKR.

\section{Constructing families of non-pure-EKR complexes}
\label{sec:examples}

In this section we explore generalizations of the counterexamples in \Cref{ssec:limits}.

\begin{proposition}
\label{lem:join}
  Let~$\C$ be a pure complex with~$m$ facets and let~$\D$ be a pure complex with~$n$ facets.
  Let~$\family$ be an intersecting family of facets of~$\C$ and let $\smallfamily$ be the star of a vertex of~$\D$.
  Assume that~$\family$ and~$\smallfamily$ both have maximal cardinality.
  \begin{itemize}
    \item If~$\C$ is not        pure-EKR and $n|\family| >   m|\smallfamily|$, then the simplicial join $\C \ast \D$ is not        pure-EKR.
    \item If~$\C$ is not strict pure-EKR and $n|\family| \ge m|\smallfamily|$, then the simplicial join $\C \ast \D$ is not strict pure-EKR.
  \end{itemize}
\end{proposition}

\begin{proof}
  The vertex stars in $\C \ast \D$ are of the form $\st_\C (v) * \D$ for $v \in \vertices(\C)$ or $\C \ast \st_\D (w)$ for $w \in \vertices(\D)$. The
   latter have size $\le m|\smallfamily|$ and the former have size $\le n|\family|$ (and $<n|\family|$ if~$\C$ is not pure-EKR).
   On the other hand, the join of~$\family$ with~$\D$ forms an intersecting family of size $n|\family|$.
\end{proof}

\begin{corollary}
  \label{cor:nonEKR}
  If there exists a flag~$d$-manifold~$\C$ with $d \geq 0$ which is not pure-EKR, then for every $\ell \geq 2$
  there are infinitely many flag $(d+\ell)$-manifolds which are not pure-EKR.
\end{corollary}

\begin{proof}
  Let~$\family$ be an intersecting family of facets of~$\C$ with strictly more facets than  any vertex star.
  For $\ell=2$, let~$\D$ be the boundary of an~$n$-gon with~$n$ sufficiently large such that $n |\family| > 2 |\C|$.
  Then \Cref{lem:join} implies $\C\ast\D$ is not pure-EKR.
  For $\ell=3$ we apply the same idea, taking as~$\D$ any flag triangulation of the~$2$-sphere with~$n$ triangles and maximum vertex degree $6$, such that $n|\family| > 6|\C|$.
  Then we obtain again that $\C\ast\D$ is not pure-EKR.
  For $\ell\ge 4$ we apply induction on $\ell$.
\end{proof}

\begin{remark}
  Note that the construction in the proof also shows that the gap between the maximum size of an intersecting family and the maximum size of a star can be made arbitrarily large.
  \end{remark}

A result similar to \Cref{cor:nonEKR} holds for strict pure-EKR: from any flag $d$-manifold that is not strict pure-EKR and for any $\ell \geq 2$ one can construct infinitely many  flag $(d+\ell)$-manifolds which are not strict pure-EKR. However, in the non-strict case we can give a more explicit statement:

\begin{corollary}
  \label{cor:nonEKR-2}
  The join of the boundary of an octahedron with any flag $d$-manifold is not strict pure-EKR. In particular, there are infinitely many non-strict pure-EKR flag manifolds in every dimension $\ge 4$.
\end{corollary}

\begin{proof}
Let~$\C$ be a flag $d$-manifold and $\C'$ its join with the boundary of an octahedron. It follows that $|\C'|=8|\C|$. As in the proof of \Cref{lem:join}, there are two types of stars in $\C'$. The join of a star in the octahedron with the whole of~$\C$ has size $4|\C|=\frac12|\C'|$. And the join of a vertex $v$ of~$\C$ with the whole octahedron has size $8 |\st_\C (v)|$ which, by \Cref{prop:flaghalf} below,
is at most $8 \frac12 |\C|=\frac12 |\C'|$.

But the join of~$\C$ with the alternating set of facets of the octahedron is also an intersecting family of size $\frac12 |\C'|$, and $\C'$ is not strict pure-EKR.
\end{proof}

\begin{proposition}
\label{prop:flaghalf}
  Let~$\C$ be a flag~$d$-manifold, and let $v \in \vertices(\C)$. Then $ |\st_\C (v)| \leq \frac12 |\C|$ with equality if and only if~$\C$ is the double suspension over $\link_\C (v)$.
\end{proposition}

\begin{proof}
  Since~$\C$ is flag, every edge of~$\C$ outside $\st_\C (v)$ must have at least one endpoint not in  $\vertices(\st_\C (v))$. It follows that, for every boundary $(d-1)$-face $\face \in \st_\C (v)$, the unique facet of $\C \setminus \st_\C (v)$ containing $\face$ cannot intersect the star of~$v$ in another boundary $(d-1)$-face. Hence, for every facet in $\st_\C (v)$ there exist a facet in the complement $\C \setminus \st_\C (v)$ and $ |\st_\C (v)| \leq \frac12 |\C|$.

  The case of equality occurs if all facets of~$\C$ contain exactly one facet of $\link_\C(v)$. In particular, all facets of~$\C$ have all but one vertex contained in $\vertices(\link_\C(v))$ and thus every edge in~$\C$ must intersect $\link_\C(v)$. Hence, every vertex in~$\C$ outside $\link_\C(v)$ has a vertex link in $\link_\C(v)$ and~$\C$ must be the double suspension of $\link_\C(v)$.
\end{proof}

\begin{remark}
In \Cref{cor:nonEKR,cor:nonEKR-2}, if we look at flag pure complexes without boundary (dropping the conditions that they are manifolds) then we can reduce dimension by one; that is, we can have $\ell \ge 1$ and dimension $\ge 3$, respectively.
The reason for this difference is that, although in dimension zero there are already infinitely many flag pure complexes without boundary (any discrete set of at least two points is one such), only the $0$-sphere (i.e., two points) is a manifold.
\end{remark}

Let us denote $\partial \beta^{d+1}$ the boundary of the $(d+1)$-cross-polytope, which is pure, flag, and pure-EKR, but not strict pure-EKR, as shown in \Cref{ex:counterex1}. It is also the smallest flag-manifold in every dimension, and the simplest example of
\Cref{prop:flaghalf}, since a join of cross-polytopes is a cross-polytope (in facet, $\partial \beta^{d+1}$ is the join of $(d+1)$-copies of the $0$-sphere).
\Cref{ex:counterex1} moreover states that $\partial \beta^{4}$ is not $2$-intersecting pure-EKR. We now extend this remark  to other parameters:

\begin{counterexample}[$t$-intersecting pure-EKRness for cross polytopes]
\label{exm:crosspolytopes}
  The boun\-dary $\partial \beta^{d+1}$ of the cross polytope $\beta^{d+1}$ is not $t$-intersecting pure-EKR for $2 \leq t \leq d-1$:

  First, any facet together with its $d+1$ neighbors form a $(d-1)$-intersecting family.
  On the other hand, the star of each face of size $(d-1)$ (i.e., the star of each $(d-2)$-face) has four facets.
  Thus, for $d \geq 3$, $\partial \beta^{d+1}$ is not $(d-1)$-intersecting pure-EKR.
  In particular, the boundary of the $4$-dimensional cross polytope is not~$2$-intersecting pure-EKR.

  We conclude with the same argument as in the proof of \Cref{lem:join}. That is, if $\partial \beta^{d+1}$ is not $t$-intersecting pure-EKR then also $\partial\beta^{d+2}$ is not $t$-intersecting pure-EKR.
\end{counterexample}

We finish with a generalization of the flag complex with boundary shown not to be pure-EKR in \Cref{exm:elementary}. Such examples illustrate the necessity of assuming ``without boundary'' in \Cref{conj:manifolds,conj:no_free_ridges}.

\begin{itemize}
    \item For $d\geq 2$, the~$d$-simplex~$\C$ together with the $(d+1)$ simplices of dimension $d$ adjacent to~$\C$ is not pure-EKR.
    \item For $d\geq 3$, a triangulated~$d$-ball consisting of two tetrahedra with a common $(d-1)$-simplex $\face$ together with a sufficiently large number of~$d$-simplices around each of the $(d-2)$-faces of $\face$ (i.e., the union of the sufficiently large stars of the $(d-2)$-faces of a $(d-1)$-face) is not pure-EKR.
    \item For $d \geq k$, the union of the (sufficiently large) stars of all $(d-k+1)$-faces of a $(d-k+2)$-face is not pure-EKR.
\end{itemize}

Note that these complexes cannot be extended to flag~$d$-manifolds without boundary which are not pure-EKR.
To see this, observe that closing the boundary while preserving the flagness means increasing the number of facets in the vertex stars of boundary vertices by at least the number of boundary ridges containing the given vertex.
In particular, applying this operation in the examples above necessarily yields pure-EKR complexes.


\bibliographystyle{plain}
\bibliography{bibliography}

\end{document}